\documentclass[12pt]{article}
\usepackage{style}
\usepackage{booktabs}

\usepackage{array}

\DeclareMathOperator{\Ver}{Ver}
\title{\MakeUppercase{Fixed point Floer cohomology of Dehn twists I: splitting formulas}}
\date{\today}

\author{Maxim Jeffs, Yuan Yao, and Ziwen Zhao}

\begin{document}
\maketitle

\section*{Abstract}

This paper is the first in a series following on our earlier work \cite{yao-zhao,JYZ} studying the pair-of-pants product on fixed point Floer cohomology. In \cite{yao-zhao,JYZ} we fully computed this product for Dehn twists on surfaces of genus $g\geq 2$, and used it to compute a version of the (small) `quantum cohomology' for nodal curves. In the present work, we develop tools for computing the fixed point Floer cohomology and the associated product in the case of Dehn twists in all higher dimensions: for iterated Dehn twists around a Lagrangian sphere in a Liouville domain, we show that the product and differential on the fixed point Floer cohomology split into local and Morse-theoretic contributions on the level of cochains, using some new confinement results for $J$-holomorphic curves. The local contributions are expected to recover a finite sector of the homology of (twisted) loop spaces of $S^n$ along with an associated Chas-Sullivan product, which we will examine in detail in future work. We also discuss some immediate applications and curiosities for future work.

\tableofcontents

\section{Introduction}
Fixed point Floer (co)homology (also known as symplectic Floer (co)homology) is a Floer theory associated to a symplectomorphism. It has been a key tool in symplectic geometry since it was first introduced by Andreas Floer in 1988 \cite{Floer}, and has since become an object of study in its own right. Though originally introduced to prove the Arnol'd conjecture for Hamiltonian symplectomorphisms (in this setting it is also referred to as Hamiltonian Floer (co)homology), its extension to all symplectomorphisms has featured prominent applications such as distinguishing exotic symplectomorphisms that are smoothly isotopic but not symplectically isotopic to the identity \cite{seidel_thesis}.

Fixed point Floer cohomology has been computed for all surface symplectomorphisms \cite{dostoglou1994self,pozniak1994floer,seidel1996symplectic,gautschi2002floer,eftekhary2002floer,cotton2009symplectic}, and there are a few computations in higher dimensions \cite{seidel_thesis, Pedrotti, poza2025, Mclean_log}\footnote{Many of these constructions work by finding clever geometric settings where there are no/few $J$-holomorphic curves one needs to consider.}.
Although it was known early on that these fixed point Floer cohomology groups admitted a `pair of pants product' (including equivariant versions \cite{equivariant_pants}), there were very few understood examples of this product structure beyond the setting of Hamiltonian symplectomorphisms; and it was only recently that the additional information contained in these products has been computed and exploited \cite{yao-zhao,JYZ}.  

In \cite{yao-zhao,JYZ}, we showed that if $\phi$ is a Dehn twist on a genus $g\geq 2$ surface $\Sigma_g$, then 

\begin{theorem}
Under the pair of pants product there is a ($\ZZ_2$-graded) ring structure on the direct sum $\bigoplus_{d=0}^{\infty} HF^*(\Sigma_g,\phi^d)$, and we have isomorphisms of $\CC$-algebras:
\[
\bigoplus_{d=0}^{\infty} HF^0(\Sigma_g,\phi^d) \cong (\CC[X,Y,Z]/(XYZ-Y^3-Z^2) ) \oplus \mathbb{C}
\]
where $X,Y,Z$ all act trivially on the second $\mathbb{C}$ factor.
The degree-$1$ part is a module over the above ring, and we have a module isomorphism:
\[
 \bigoplus_{d=0}^\infty HF^1(\Sigma_g,\phi^d) \cong 
\mathbb{C}[X,Y,Z]/(XYZ-Y^3-Z^2) \oplus (\mathbb{C}[X])^{2g-2} \oplus \mathbb{C}.
 \]
where $HF^0(\Sigma_g,\phi^d)$ acts on the first factor by multiplication,  on the second factor by projection to $\CC[X]$ followed by multiplication, and on the last factor by projection to $\CC$ followed by multiplication.
\end{theorem}

There are infinitely many interesting $J$-holomorphic curves that contribute to this rich algebraic structure, which we managed to understand and enumerate in \cite{yao-zhao,JYZ}. Furthermore, in \cite{JYZ} we used these computations to compute direct limits of fixed point Floer cohomology:

\begin{theorem} After taking a direct limit along multiplication by the Seidel class $S \in HF^0(\phi)$ (see \cite[\S 4.2]{JYZ}), we have an isomorphism of algebras:
\[\lim_{d\rightarrow \infty} HF^0(\phi^d) \cong \mathbb{C}[Y,Z]/(YZ-Y^3-Z^2)\]
and as a module over the above, we have an isomorphism:
\[
\lim_{d\rightarrow \infty} HF^1(\phi^d)\cong \mathbb{C}[Y,Z]/(YZ-Y^3-Z^2)\oplus \mathbb{C}^{2g-2}
\]
where the direct limit of $HF^0$ acts the first factor by multiplication, and on the second factor by projection to $\CC$ followed by diagonal multiplication.
\end{theorem}

We showed in \cite{JYZ} that the above algebras could be considered as a version of (small) `quantum cohomology' for singular curves by explicitly verifying an instance of closed-string mirror symmetry. Our method provides a way of defining `genus zero invariants' of some singular varieties when na\"ive enumerative methods would produce trivial answers (as in the case of curves of genus $g \geq 2$). 

This paper is the next part in our continuing program of understanding the product structure on fixed point Floer cohomology. In this paper, we make the important first step towards computing the full algebra structure on $\bigoplus_{d=0}^{\infty} HF^{\ast}(M^{2n},\phi^d)$ where $M$ is a Liouville domain, and $\phi$ is a Dehn twist around a Lagrangian sphere.

The key result used in our computations in \cite{JYZ} was a non-crossing result for $J$-holomorphic curves from \cite{yao-zhao} that split the calculation of both the differential and the pair-of-pants product into a Morse-theoretic contribution away from the support of the Dehn twist, as well as a local calculation in the Dehn twist region, which can be identified with a finite sector of the symplectic cohomology of $T^{\ast} S^1$.  In other words, by confining the $J$-holomorphic curves, we see that the chain complex of fixed point Floer cohomology is built out of the cochain complex for Morse cohomology, coupled non-trivially with a finite sector of the cochain complex for symplectic cohomology of $T^*S^1$, with the product structure respecting this decomposition on the cochain level. The final computation falls out of carefully enumerating curves in the twist region (using our understanding of the symplectic cohomology of $T^*S^1$ on the cochain level).

In this paper, we make the crucial first step in carrying out the above computations for Dehn twists in higher dimensions: in all dimensions, we establish the analogous confinement result for $J$-holomorphic curves for all iterations of Dehn twists around Lagrangian spheres in Liouville domains. We show the fixed point Floer cohomology is built out of the Morse cochain complex of the complement of the Lagrangian sphere, coupled with a finite sector of the fixed point cochain complex on $T^*S^n$ (which we call the `wrapped Dehn twist')\footnote{Unlike earlier computations in higher dimensions, there are many interesting $J$-holomorphic curves in this sector, for both the product and the differential.}, and the product respects this decomposition of cochain complexes. The confinement methods for $J$-holomorphic curves draw on several different tools: for an exposition see \S \ref{sec:techniques}.

In the next phase of our project (currently in progress), we will carry out careful cochain-level computations of $J$-holomorphic curves that appear in the Dehn twist region. This will involve leveraging and further developing existing tools (due to Diogo and Lisi \cite{Lisi_Diogo_split,Lisi_Diogo_complement}) that give explicit enumerations of $J$-holomorphic curves that appear in Floer theory.

However, already with just the confinement of holomorphic curves understood, we can read off the fixed point Floer cohomology groups of all even iterations of Dehn twists in dimension 4 (with $\ZZ_2$ coefficients), which was not known. Inspecting the rank of this group reveals very interesting information: we point to several applications, conjectures, and possible directions for future work.

\subsection{Main theorems}

Here we introduce the main theorems of this paper. For greater readability, we do this with minimal setup: the more precise statements of these theorems will be found in the body of the paper.

Let $\phi$ be a model Dehn twist around a Lagrangian sphere in a Liouville domain $(M^{2n},d\theta_M)$ for $n \geq 2$, and let $Y_{\phi^d}$ denote the mapping torus equipped with the standard stable Hamiltonian structure. We can decompose $Y_{\phi^d}$ into two regions, referred to as the \textit{Dehn twist region}  $Y_{tw}$ and the \textit{Morse region} $Y_{Mo}$. Here $Y_{tw}$ contains the support of the Dehn twist $\phi^d$ (before we perturb $\phi$ to be nondegenerate), and $Y_{Mo}$ is the complement. 

We arrange these two regions to overlap in a small intersection region that is diffeomorphic to $[1-\epsilon,1+\epsilon]\times S^*S^n$ (the unit cosphere bundle of $S^n$) for some $\epsilon>0$. We will call this region the \textit{intersection region}. After perturbing $\phi$ to be nondegenerate, there are generators (which belong to both Dehn twist and Morse regions) in the intersection region. The generators in the Dehn twist (resp. Morse) but not in the intersection region will be referred to as generators in the \textit{strict Dehn twist region} (resp. \textit{strict Morse region}).

Our first main theorem is the following:
\begin{theorem}[Theorem \ref{thm:confinement_differential}]
After suitably perturbing $\phi^d$ to be nondegenerate, we can find a large class\footnote{In this paper the conditions we impose on $J$ are near the intersection between the Dehn twist region and the Morse region; we are free to choose $J$ arbitrarily away from these regions. By saying there exists a `large class' of such $J$s, we mean that we can satisfy all necessary transversality conditions we need within this class. In this paper we will use the term `regular' as a condition on $J$ to mean that all the $J$-holomorphic curves relevant to us are transversely cut out.} of almost complex structures on $\mathbb{R} \times Y_{\phi^d}$ and associated Floer data so that confinement of $J$-holomorphic sections is achieved: 
\begin{enumerate}
\item If a $J$-holomorphic section has both punctures asymptotic to orbits in the Morse region, then it is entirely contained in the Morse region;
\item If a $J$-holomorphic section has both punctures asymptotic to orbits in the Dehn twist region, it is entirely contained in the Dehn twist region;
\item If a $J$-holomorphic section has a puncture asymptotic to an orbit in the strict Morse region, then it is entirely contained in the Morse region. If a $J$-holomorphic section has a puncture asymptotic to an orbit in the strict Dehn twist region, then it is entirely contained in the Dehn twist region.

\end{enumerate}

\end{theorem}

Let $(CF^*(\phi^d), d)$ denote the fixed point Floer cochain complex for the iterated Dehn twist: recall that for fixed points $x$ and $y$, the  pairing $\langle dx, y \rangle$ is given by a count of $J$-holomorphic sections of $\RR \times Y_{\phi^d}$ connecting $x$ and $y$ (for our precise setup, see \S 2.4). The above result implies a particular splitting formula for the differential $d$. To state it precisely, consider the subcomplex $(CF^*_{tw}(\phi^d), d_{tw})$ which consists of generators in the twist region and where $d_{tw}$ counts $J$-holomorphic curves entirely contained in the twist region (this is a subcomplex by the above theorem); as well as $(CF^*_{Mo},d_M)$, which is a cochain complex defined analogously for the Morse region. We have a decomposition formula for $d$ in terms of the chain complexes above.

\[
\langle dx,y \rangle =\begin{cases}
\langle d_{tw} x, y\rangle \quad\text{if $x$ is in the Dehn twist region,}\\
\langle d_{Mo}x,y\rangle \quad\textup{if $x$ is in the Morse region}
\end{cases}
\]
We note that for $x$ and $y$ both in the intersection region, the two differentials agree
\[
\langle d_{twist} x, y\rangle = \langle d_{Mo}x,y\rangle
\]
and this count is given by a count of Morse flowlines in the intersection region, so it is explicitly computable. Hence if we write $d_{int}$ for the differential given by
\[
\langle d_{int}x,y \rangle = \langle d_{twist}x,y\rangle = \langle d_{Morse} x,y\rangle
\]
for $x,y$ in the intersection region, and zero otherwise, we can write
\[
d= d_{Morse}+d_{twist} -d_{int}
\]
In other words, the total cochain complex is given by the chain complex in the Morse region coupled nontrivially to the cochain complex in the Dehn twist region.

The differential in the Morse region is easily understood because it is counting Morse flow lines, but the subcomplex $CF^*_{tw}(\phi^d,d_{tw})$ is more interesting to understand. We realize it as a finite-energy subcomplex of another interesting Floer homology group on $T^*S^n$.

\begin{definition}
Take $\mu$ to be the length of the fiber coordinate on $T^*S^n$ induced by the round metric on $S^n$. 
Let $\phi$ be a (model) Dehn twist on $T^*S^n$ which is supported in $\mu\leq 1$: we define the wrapped Dehn twist $\phi^d_\infty$ as follows: consider a Hamiltonian $H$ supported away from the support of the Dehn twist, and equal to $\frac{1}{2}\mu^2$ away from the Dehn twist region. Let $\phi_{H}$ be the Hamiltonian flow of $H$. We then define the wrapped Dehn twist by $\phi^d_\infty:= \phi_H \circ \phi^d$ (we may perturb it slightly to make it nondegenerate).\footnote{In \cite{Uljarevic}, Uljarevic defines fixed point Floer homology of automorphisms of a Liouville domain using linear Hamiltonians at infinity. We expect our definition to be equivalent to theirs after a direct limit.}
\end{definition}

We observe that for even $d$ this is just the symplectic cohomology of $T^*S^n$, which we can identify with the homology of the free loop space of $S^n$ (see \cite{abouzaid_viterbo, viterbo2018,Ab_schwarz_1,Salamon_Weber,AbSchwarz}). The wrapped Dehn twists for odd $d$ are more interesting: we recall that the twisted loop space (with respect to the antipodal map) of $S^n$ is given by
\[
\mathcal{L}_{tw}(S^n):=\{ \text{continuous maps} \; \gamma:[0,1]\rightarrow S^n \; | \; \gamma(0)=-\gamma(1)\}.
\]
We pose this as a question for future work:
\begin{conjecture}(Twisted Viterbo isomorphism)\footnote{This conjecture was known to Yannis Bähni since at least 2023.}
Up to shifts, $HF(\phi^n_\infty)$ can be identified with $H_*(\mathcal{L}_{tw}(S^n)).$
\end{conjecture}

We observe that if we set up the Hamiltonian and almost complex structure appropriately, we see from the maximum principle that the chain complex $CF^*_{tw}(\phi^d,d_{tw})$ can be viewed as a subcomplex of $CH^*(\phi^d_\infty)$. In topological terms, we are claiming the fixed point Floer cohomology complex for Dehn twists is the coupling of a Morse cochain complex with another cochain complex that computes the homology of the loop space (resp. conjecturally the twisted loop space for odd $d$) up to a certain action level.

We next describe the confinement principle for the product structure on fixed point Floer cohomology. The product we are interested in maps
\[
HF^*(\phi^m)\otimes HF^*(\phi^n)\rightarrow HF^*(\phi^{m+n})
\]
and is defined by counting $J$-holomorphic sections of a symplectic fibration $E^{m,n}_
\delta\rightarrow B_3$, where $B_3$ is a sphere with three punctures: see Section \ref{sec:prod_str_Dehn_twists} for a careful construction.

We can decompose this fibration into two pieces, each with boundary: firstly, we have $E_{tw}^{m,n} \rightarrow B_3$ which is a nontrivial bundle where the Dehn twists are supported (called the \textit{Dehn twist region}), and secondly, $E^{m,n}_{Mo}\rightarrow B_3$ where the fibration is trivial (called the \textit{Morse region}): curves in this region are simply computing the cup product in Morse cohomology. The two regions overlap along a thin region diffeomorphic to
\[
(1-\epsilon, 1+\epsilon) \times S^*S^n \times B_3\rightarrow B_3
\]
for some $\epsilon>0$, which we shall also call the \textit{intersection region}.

Our confinement principle for $J$-holomorphic sections computing the product structure says that:

\begin{theorem}[Section \ref{sec:confinement_Morse_Bott}, Proposition \ref{prop:confinement_compatible_Morse}: Theorem \ref{thm:tame_to_compatible}]\label{thm:product_confinement}
We can find a large class of almost complex structures $J$ and associated Floer data on $E^{m,n}_\delta$ so that confinement is achieved: suppose $u$ is a $J$-holomorphic section of $E^{m,n}_{\delta}$, then:
\begin{enumerate}
\item If $u$ has a puncture asymptotic to an orbit in the strict Dehn twist region, then $u$ is contained in the Dehn twist region $E_{tw}^{m,n}$;
\item If $u$ has a puncture asymptotic to an orbit in the strict Morse region, then $u$ is contained in the Morse region $E^{m,n}_{Mo}$;
\item If all three punctures of $u$ are asymptotic to orbits in the intersection region, then $u$ must be contained in the intersection region $(1-\epsilon, 1+\epsilon) \times S^*S^n \times B_3$ as well.
\end{enumerate}

\end{theorem}
Using the above we can see that we can decompose the product into
\[
\bullet_{tw} : CF^*(\phi^m,Y_{tw})\otimes CF^* (\phi^n, Y_{tw}) \rightarrow CF^*(\phi^{m+n},Y_{tw})
\]
where the cochain generators are the orbits in the Dehn twist region, and the product counts $J$-holomorphic sections contained in $E^{m,n}_{tw}$: we write this product as $x\bullet_{tw}y$. Similarly, there is a Morse product given by
\[
\bullet_{Mo} : CF^*(\phi^m,Y_{Mo})\otimes CF^* (\phi^n, Y_{Mo}) \rightarrow CF^*(\phi^{m+n},Y_{Mo})
\]
where the cochain complex is generated by orbits in the Morse region and the product simply computes the cup product on Morse cohomology. 
If $x,y,z$ are all in the intersection region, then
\[
\langle x \bullet_{tw} y,z\rangle = \langle x\bullet_{Mo} y,z\rangle
\]
We will also call this the \textit{restricted product} $\bullet_{int}$ and we set it equal to zero unless $x,y,z$ are all in the intersection region. This product can be computed only using Morse theory. Using this we can give a chain-level formula for $*$:
\[
x\bullet y = x\bullet_{tw} y + x\bullet_{Mo}y - x\bullet_{int} y.
\]

The product in the twist region $E_{tw}^{m,n}$ is a finite sector of a bigger product structure, i.e., it is a finite sector of the product structure on fixed point Floer cohomology of wrapped Dehn twists:
\[
CF^*(\phi_\infty^n)\otimes CF^*(\phi^m_\infty) \rightarrow CF^*(\phi^{n+m}_\infty).
\]
If we set up the Hamiltonian carefully, we can find action levels $L_n, L_m, L_{n+m}$ such that for filtered chain complexes $CF^{\leq L_*}(\phi^*_\infty)$,  there is a canonical bijection of generators and pair of pants in the Dehn twist region and the generators and curves counted by
\[
CF^{\leq L_n}(\phi_\infty^n)\otimes CF^{\leq L_m}(\phi_\infty^m)\rightarrow CF^{\leq L_{n+m}}(\phi_\infty^{m+n}).
\]
Observe for even $m$ and $n$, the cohomology-level product 
\[
HF(\phi_\infty^n)\otimes HF(\phi_\infty^m)\rightarrow HF(\phi^{m+n}_\infty)
\]
can be identified with the Chas-Sullivan product on the homology free loop space of $S^n$ \cite{abouzaid_viterbo,AbSchwarz}. We conjecture there is a Chas-Sullivan type product on twisted loop spaces of $S^n$ that matches the fixed point Floer cohomology product $HF(\phi_\infty^n)\otimes HF(\phi_\infty^m)\rightarrow HF(\phi^{m+n}_\infty)$, where $m$ and $n$ may be odd.\footnote{In our forthcoming work, we will compute the fixed point Floer cohomology of all wrapped Dehn twists and the associated product structures (on the cohomology level).}

While in general more work is required to understand the \textit{chain-level} interaction between the Morse and Floer-theoretic contributions to the differential and product in the neck regions, if we restrict our attention to dimension 4 and $\ZZ_2$ coefficients, we may be lucky by looking at the grading and counting the number of fixed points vs dimension of cohomology groups in a given degree. In particular, by comparison to the work of Diogo-Lisi \cite{Lisi_Diogo_complement} on symplectic cohomology,  we may obtain the following result\footnote{There are still many interesting curves in the twist region, but for geometric reasons they appear in pairs. We see a much richer cochain complex if we use $\mathbb{Z}$ coefficients, for example.}:

\begin{theorem} \label{thm:even_dim_4}
Let $(M^4,d \theta_M)$ denote an exact symplectic 4-manifold and let $\phi$ denote a Dehn twist. Then\footnote{Implicit in this calculation is how to perturb $\phi$ near the boundary of $M$, which we will specify in the proof.} for $n\geq1$
\[
HF^*(\phi ^{2n},\ZZ_2) = H^*(M\setminus S^n) \oplus \ZZ_2^2 \oplus \ZZ_2^{4(n-1)}\]
\end{theorem}

We will explore this further in future work (see \S \ref{subsec:new_directions}).

\subsubsection{A general splitting criterion}\label{subsec:universal}

We make the curious observation that in our proof of confinement of $J$-holomorphic curves, we did not use too many properties about the Dehn twist itself; most of the relevant analysis was carried out in the Morse region. Hence, we would like to formulate a general criterion for compactly supported symplectomorphisms for which we can compute their fixed point Floer cohomology by decomposing the cochain complex into several pieces, which is of independent interest.

Suppose $(M,d\theta_M)$ is a Liouville manifold and let $\phi$ denote a compactly supported symplectomorphism. 

We suppose there is a (connected) contact-type hypersurface $Z\subset M$ with the contact form given by the restriction of $\theta_M$, and there is a slightly thickened neighborhood of $Z$ so that $(-\epsilon,+\epsilon)_r\times Z$ on which $\theta_M = e^r\theta_M|_Z$. We assume $Z$ decomposes $M$ into two regions, $U_1$ (which intersects $r<\epsilon$ and where the support of $\phi$ lies), and $U_2$ (which intersects $r>-\epsilon$, which is analogous to the Morse region).

We assume that $\phi$ preserves the hypersurface $Z$, i.e. it maps $U_2$ to $U_2$ and $U_1$ to $U_1$.

Form the mapping torus $Y_\phi$ with the standard stable Hamiltonian structure $(dt,\omega)$. We assume the following:
\begin{enumerate}
\item Suppose further that $\phi$ is an exact symplectomorphism, i.e., $\phi^*\theta_M = \theta_M+dF$ for some $F$: then there is a $1$-form $\lambda$ so that on $Y_\phi$ we have $\omega =d\lambda$. Then each periodic orbit $\gamma$ has action $\mathcal{A}(\gamma) = \int_\gamma \lambda$. 
\item We can decompose $Y_\phi = U_{1,\phi} \cup (S^1\times U_2)$, where $U_{1,\phi}$ is the mapping torus of $U_1$. The regions  $U_{1,\phi}$ and $(S^1\times U_2)$ intersect in $S^1\times (-\epsilon,\epsilon)\times Z \subset S^1\times U_2$.
The restriction of $\lambda$ to $S^1\times U_2$ takes the form $\theta_M +f_1dt$. Here $f_1$ is a Morse function on $U_2$ (away from the $r=0$ locus, which is Morse-Bott), it is of the form $\frac{1}{2}\epsilon_Gr^2$ in $(-\epsilon,\epsilon)\times Z$ and has value between $\frac{1}{2}\epsilon_G \epsilon^2$ and $\frac{3}{2}\epsilon_G \epsilon ^2$ on the rest of $U_2$, for some $\epsilon, \epsilon_{G} > 0$.
\item We can take $\epsilon_G$ and $\epsilon$ as small as we please, but uniformly in $\epsilon,\epsilon_G$ we can find a $C$ large so that in a region near the intersection of $U_{1,\phi}$ and $S^1\times U_2$, which we take to be of the form
\[
U_C:=S^1\times [-1/C,\epsilon]_r \times Z
\]
we have $\lambda =\theta_M +H(r)dt$. Note $U_C$ is just a slight extension of the chart we have chosen for $U_{1,\phi}\cap S^1\times U_2 = S^1\times (-\epsilon,\epsilon)\times Z \subset U_C$. Here $H(r)$ satisfies:
\begin{itemize}

\item $H$ only depends on $r\in (-1/C,\epsilon)$, and is a smooth convex function. For $r\in (-\epsilon,\epsilon)$ we have $H(r)=\frac{1}{2}\epsilon_Gr^2$.
\item $H(r)$ does not have any time-1 periodic orbits in the interval $[-1/C,0)$,
\item $H(r)$ is independent of $\epsilon,\epsilon_G$ in the interval $(-1/C,-1/C^2)$.
\end{itemize}
\end{enumerate}

\begin{theorem}
Assume $\phi$ satisfies the above assumptions: then we can find $\epsilon, \epsilon_G, \delta>0$ sufficiently small such that after a size $\delta$ perturbation of $\phi$ to ensure nondegeneracy, we can choose a compatible $J$ of confinement type  so that the splitting formula holds. i.e. we can write \[d= d_{u_1}+d_{Mo}-d_{int}.\] Here $d_{Mo}$ counts $J$-holomorphic curves in $U_2$, which can be identified with Morse homology, $d_{u_1}$ counts $J$-holomorphic curves contained in $U_1$, and $d_{int}$ counts $J$-holomorphic curves in $(-\epsilon,\epsilon)\times Z$, which is also computable via Morse theory.
\end{theorem}

The proof is verbatim the proof of confinement we gave for the differential for fixed point Floer cohomology of Dehn twists in Section \ref{sec:confinement_differential}. The requirements we place on $J$ are only in the $r\in (-1/C,\epsilon)$ region near where $U_{1,\phi}$ intersects $S^1\times U_2$, and $J$ is allowed to be arbitrary elsewhere.

There are many other known compactly supported symplectomorphisms besides Dehn twists, for example fibered Dehn twists as well as twists around other kinds of Lagrangians, such as Lagrangian $\mathbb{CP}^n$s. We expect our splitting criteria to apply to large families of examples of these types: thus understanding fixed point Floer cohomology of these symplectomorphisms is reduced to understanding a chain-level model of the $J$-holomorphic curves in the $U_1$ region. This is by no means easy, and one would need to construct more general tools for this purpose. For Dehn twists around Lagrangian $\mathbb{CP}^n$s, we may again leverage the Viterbo isomorphism to relate this to string topology. For fibered Dehn twists the story is less clear, but is an exciting area for future work.

\subsection{Applications, observations, and future directions}\label{subsec:new_directions}

We may regard the result of Theorem \ref{thm:even_dim_4} as proving a version of an Arnol'd-type conjecture for even iterations of Dehn twists in dimension 4. It is known that these maps $\phi^{2n}$ are smoothly isotopic to the identity, and for smooth diffeomorphisms the general lower bound for the number of fixed point is given by the Lefschetz fixed point theorem. However we observe that for nondegenerate perturbations of Dehn twists, the number of fixed points is bounded below by the rank of the fixed point Floer cohomology (which increases with the number of iterations). We believe this bound on the number of fixed points is more or less sharp, because we can construct nondegenerate Dehn twists whose number of generators in the Dehn twist region is exactly the rank given by Theorem \ref{thm:even_dim_4}.

The world has gone quite far in understanding dynamics since the Arnol'd conjecture.\footnote{The fact that the rank of Floer cohomology grows with the number of iterations is new and is in contrast with the Hamiltonian Floer homology. We invite the dynamics community to figure out the right dynamics questions to ask and answer for this class of symplectomorphisms.}  In particular, it is known that the Arnol'd conjecture fails in the $C^0$ setting \cite{C0_counterexample}: there are continuous Hamiltonian homeomorphisms of symplectic manifolds with only one fixed point. Hence we find it interesting to ask:

\begin{question}
Is there a $C^0$ symplectomorphism in the ($C^0$ closure of) the connected component of $\phi^{2n}$ that has fewer fixed points than the rank of the fixed point Floer cohomology? If so, what is the minimal number of such fixed points?
\end{question}

In another direction concerning fixed points, if we compare the rank of the Floer cohomology for $\phi^{2n}$ and $\phi^{2n+2}$, even though they are both smoothly isotopic to the identity and not symplectically isotopic, $\phi^{2n+2}$ is ``further away'' from the identity component than $\phi^{2n}$ since it has more mandatory fixed points. In other words, this would suggest: 

\begin{question}
Is there an appropriate notion of \emph{hierarchy} on the symplectic mapping class in which we can make sense of the statement ``$\phi^{2n+2}$ is further away from the identity component than $\phi^{2n}$''?
\end{question}

This kind of hierarchy has been observed in contact manifolds, in the context of fillability of contact manifolds \cite{algebraic_torsion,hierarchy,landscape}. A contact manifold is (exactly) fillable if it is the positive boundary of a Liouville domain. For two contact manifolds $Y_1$ and $Y_2$ that are not fillable, one may still be ``less'' fillable than the other. This is made precise by the various kinds of torsion type invariants. For definiteness, we will take planar torsion, which is an integer-valued invariant defined geometrically. If both $Y_1$ and $Y_2$ have nonzero planar torsion, then they are not fillable. But if $Y_1$ has higher planar torsion than $Y_2$, then there may exist (exact) symplectic cobordisms from $Y_1$ to $Y_2$, but symplectic cobordisms from $Y_2$ to $Y_1$ are forbidden. Hence, in that sense $Y_1$ is ``further away'' from being fillable than $Y_2$.

To make the analogy more transparent, we will rewrite our theorems in a more suggestive way. Let $\phi^{2n}$ denote the $2n$-th iterate of a Dehn twist: the mapping torus $Y_{\phi^{2n}}$ has a stable Hamiltonian structure. The fact that $\phi^{2n}$ is smoothly isotopic (rel. boundary) to the identity means $Y_{\phi^{2n}}$ is diffeomorphic to $S^1\times M$, whereas the fact that $\phi$ has infinite order in the symplectic mapping class group should mean the stable Hamiltonian structure on $Y_{\phi^{2n}}$ is ``different'' from the the product stable Hamiltonian structure on $S^1\times M$ in some appropriate sense. The previously observed fact that $\phi^{2n+2}$ is ``less'' symplectically isotopic to the identity than $\phi^{2n}$ should also translate to some measurable geometric quantity we can associate to the stable Hamiltonian structure on the mapping tori, though currently we do not know what this may be.

If one boldly believes in an analogy between contact manifolds and mapping tori and is willing to engage in some speculation, one could ask:

\begin{question}
 Is there a reasonable cobordism theory of stable Hamiltonian structures (of mapping tori) that we can use to explain the above observed phenomenon\footnote{Is there a good notion of Giroux torsion in the mapping torus world, for example?}: that $Y_{\phi^{2n+2}}$ is ``further away'' from the product stable Hamiltonian structure $S^1\times M$ than $Y_{\phi^{2n}}$?
\end{question}
Or conversely, leveraging our optimism in the opposite direction: a way to phrase our results is that we have detected the fact that $\phi^{2n+2}$ is not (Hamiltonian) isotopic to the identity by using a count of fixed points, so we can ask:
\begin{question}
 Is there a dynamical characterization of torsion invariants for contact manifolds, or more generally in the non-fillability of contact manifolds?
\end{question}

Lastly there has been some interest in Dehn twists from the symplectic dynamics community. For certain finite energy hypersurfaces for the three-body problem, one can construct open book decompositions whose monodromy map on the pages is a squared Dehn twist \cite{three_body}. It would be very interesting to see if our approach can be pushed further to yield dynamics results.

\subsubsection{Mirror symmetry}

There are now several established schools in mirror symmetry.
One idea put forth by Abouzaid and Siebert in the early 2010s, partly motivated by theta functions (see \cite{Hulya}), is the following. If $X = M \setminus D$ is an exact symplectic manifold mirror to an algebraic variety $\check{X}$, then the (degree-zero) symplectic cohomology ring $SH^{0}(X)$ is expected to correspond to the ring of global functions on $\check{X}$: 
$$
SH^{0}(M \setminus D) \cong H^{0}(\check{X}, \scr{O}_{\check{X}})
$$
Hence when the mirror $\check{X}$ is affine, this gives a global description of the mirror that is `intrinsic' to $M \setminus D$ (cf. \cite{intrinsic_mirror}).

Of course, one may be interested in understanding mirror symmetry when $\check{X}$ is \textit{not} affine, which leaves open the question of how to recover the data that allows one to reconstruct the mirror. The Gross-Siebert program \cite{intrinsic_mirror} provides an answer: one takes a toric degeneration $\check{\scr{X}} \to \Delta$ of $\check{X}$ along with an ample line bundle $\scr{L}$ (explicitly given in terms of some piecewise linear function $\phi$). Tensoring by the line bundle $\scr{L}$ gives an automorphism of $D^b \mathrm{Coh}(\check{X})$ that corresponds under mirror symmetry to some automorphism of the Fukaya category $\scr{F}(X)$, which should in fact be induced by a symplectomorphism $\phi$ coming from the monodromy of the mirror toric degeneration: this idea seems to have been known for some time in the context of SYZ mirror symmetry (for a precise statement in the context of the Gross-Siebert program, see \cite[\S 5.3]{jeffs}). Analogous to the symplectic cohomology case, work in progress by Perutz-Siebert posits that one may recover the homogeneous coordinate ring on the mirror side by studying the symplectic monodromy ring on the $A$ side:
\begin{equation*}
    \bigoplus_{d=0}^{\infty} HF^{0}(\phi^d) \cong \bigoplus_{d=0}^{\infty} H^{0}(\check{X}, \scr{L}^{\otimes d})
\end{equation*}
Since $\scr{L}$ is ample, the homogeneous coordinate ring on the right is sufficient to recover $\check{X}$ from its embedding in projective space. We carried out this computation in detail in \cite[\S 2.4]{JYZ} in the case of curves of genus $\geq 2$, finding a number of homogeneous coordinate rings for different embeddings of nodal elliptic curves into (weighted) projective spaces. Of course, in general this monodromy symplectomorphism is not a composition of Dehn twists, but sometimes instead a composition of \textit{fibered} Dehn twists, which will be studied further in forthcoming work of Perutz-Siebert. 

One instance where we expect our computations to be useful for understanding mirror symmetry is K3 surfaces\footnote{Since K3s are hyperK\"ahler, one expects not to see (generically) any non-constant genus-zero holomorphic curves, and hence any na\"ively-defined curve-counting invariant will be trivial. By contrast, our Floer-theoretic algebra construction are expected to provide highly non-trivial answers.}. For instance, one may consider the mirror pair of nodal K3 surfaces studied in a series of papers \cite{HLLY1,HLLY2,HLLY3,HLLY4,HLLY5}, constructed as follows. Here $X$ is a nodal sextic K3 surface, obtained by taking a double cover of $\mathbb{P}^2$, branched over $D$, a collection of six lines in general position: this produces a sextic K3 surface with 15 nodal points. The mirror $\check{X}$ is likewise obtained as a double cover of a degree-$6$ del Pezzo surface $Bl_{3}(\mathbb{P}^2)$ branched over a divisor $\check{D}$ described by $\Pi_{i=1}^{3} g_{i1} g_{i2} = 0$ where $g_{i1}, g_{i2}$ are general sections of $\scr{O}(H - E_i)$. This results in a K3 surface $\check{X}$ with $12$ nodal points, and \cite{HLLY3} predicts that it is mirror to $X$: indeed, they verify Hodge-theoretic mirror symmetry for this pair (and their generalizations).

In future work, we hope to generalize these computations to our fixed point Floer cohomology model for the quantum cohomology of a singular hypersurface. What makes this computation tractable in this case is that the monodromy around the `large complex structure limit point' is given by the composition of Dehn twists around a collection of disjoint vanishing cycles. These vanishing cycles around which we perform the Dehn twists in $X$ can be understood as matching cycles in $\mathbb{P}^2$ corresponding to degenerations of the curve $D$. Combining this with a gluing result for Seidel classes from \cite[Theorem 4.8]{JYZ}, we hope to largely reduce the Floer theory to a local computation using our curve confinement results proven here.

\subsubsection{The arc-Floer conjecture}

There has been a recent surge of interest in fixed point Floer cohomology in the algebraic geometry community via the arc-Floer conjectures \cite{arc_floer_plane, poza2025}. Consider an algebraic map $f:\mathbb{C}^n\rightarrow \mathbb{C}$ with an isolated singularity at the origin. The arc space $\mathcal{X}_m$ of $f^{-1}(0)$ is the space of maps from $\mathbb{C}[t]/t^{m+1}$ to $f^{-1}(0)$. The monodromy $\phi$ around the singular fiber gives a symplectomorphism acting on the nearby fiber $f^{-1}(\epsilon)$.

The arc Floer conjecture posits that up to grading shifts\footnote{Notice that the right-hand side is Floer \emph{cohomology} in our notation. By our conventions, our cohomology is the same as homology in \cite{Uljarevic}. The $+$ indicates how we perturb $\phi$ near the boundary of $f^{-1}(\epsilon)$.} There is an isomorphism of groups:
\[
H^*_c(\mathcal{X}_m)\cong HF^*(\phi^m,+).
\]
This has been verified in complex dimension 2 (for plane curves) \cite{arc_floer_plane} and for certain kinds of semi-homogeneous singularities \cite{poza2025}. All of these computations proceed by computing the dimension of homology groups on both sides and comparing the answers.\footnote{All the computations done on the Floer side rely on McLean's spectral sequence \cite{Mclean_log}.} On one hand, we hope our confinement results (see section \ref{subsec:universal}) can give some tools for computing fixed point Floer cohomology in more novel situations, and on the other hand we hope the confinement results give some hints on why such an isomorphism should be true. 

Note that on the algebraic side, for smooth varieties the homology of the arc space is just the homology of the space itself: the interesting information captured by the arc space is in the singular locus of $f^{-1}(0)$. Our confinement results give a symplectic interpretation of this in the case where the monodromy is a single Dehn twist: Floer-theoretically all the interesting information is concentrated near the vanishing cycles of the fibration, which is the locus that collapses to the singularity of $f^{-1}(0)$.

\subsection{Techniques and context}\label{sec:techniques}

In this section, we give some context for the technical aspects of this paper and explain its relationship to other work. In \cite{yao-zhao}, we proved a confinement result of a similar nature in the case of Dehn twists on (real two-dimensional) surfaces, in which we also see the coupling of a finite sector of symplectic cohomology of $T^*S^1$ with the Morse homology of the complement. The main technical result in that paper is the extension of an intersection theory lemma from \cite{hutchings2005periodic} which considers intersections of $J$-holomorphic curves with $3$-dimensional surfaces. In \cite{hutchings2005periodic}, this lemma was used to control the $J$-holomorphic curves that appear in the differential of periodic Floer homology of Dehn twists, and we adapted that lemma to symplectic cobordisms in order to study product (and coproduct) operations and achieved the necessary confinement.

When we considered the generalization of confinement of $J$-holomorphic curves in a higher-dimensional setting, we found the most natural extension of \cite{hutchings2005periodic} actually turns out to be the integrated maximum principle of Abouzaid-Seidel \cite[Lemma 7.2]{abouzaid_seidel_viterbo}. To be precise, this lemma (and its generalization in our settings) covers about 3 out of 4 possibilities required to confine $J$-holomorphic curves; the last case is covered by a lemma of Hein \cite{Hein} and its further generalizations.

Historically, the main use of Abouzaid-Seidel's maximum principle has been to rule out holomorphic curves escaping to infinity in symplectic cohomology/wrapped Floer cohomology type theories (there were also more isolated uses of this theorem to avoid bad degenerations of $J$-holomorphic curves, for example \cite{Lisi_Diogo_complement}). As far as we are aware, the work of Ganor-Tanny \cite{Ganor_2023} was among the first to systematically use the Abouzaid-Seidel maximum principle (plus Hein's result) to try to confine $J$-holomorphic curves enough to understand the chain complex that makes up Hamiltonian Floer homology. They did not achieve a full enumeration of the $J$-holomorphic curves that appear in the chain complex, but constrained the curves enough to draw conclusions about spectral invariants in Hamiltonian Floer homology (but they did not use it to \textit{compute} the homology, which was already known; and the product was not considered).

In light of the above developments, one major conceptual contribution of this paper is the realization that the confinement results of Abouzaid-Seidel (and Hein) can be used to compute fixed point Floer cohomology of compactly supported symplectomorphisms and their corresponding product operations. While we are not in the Hamiltonian Floer homology setup of the original work of Abouzaid-Seidel or Hein, by using the language of stable Hamiltonian structures, we can set things up carefully so that a piece of the symplectic fibration in which we count $J$-holomorphic curves looks similar enough to Hamiltonian Floer theory so that the original inequalities derived in Abouzaid-Seidel and Hein can be proved also in our setting (and from which we then deduce confinement). On a technical level, the generalizations of Abouzaid-Seidel and Hein's methods in the case of products requires some further work. Because these arguments rely on action/energy, we have to carefully construct the symplectic fibrations to avoid the appearance of curvature \cite{McDuff_Salamon}. We first prove the relevant inequalities/confinement in the Morse-Bott setting where we can arrange the curvature to be zero, then deduce the confinement for the nondegenerate case via a degeneration to cascades argument. 

\subsection{List of conventions} \label{sec:conventions}

The symplectic geometry literature has a number of incompatible conventions that make it difficult to connect results from different papers. We hope this list of conventions can help the reader orient our paper among the different conventions.

\subsubsection*{Signs of Hamiltonian vector fields}

We take the convention that the tautological 1-form on $T^*S^n$ is given by $\lambda = \sum_{i=1}^n p_i dq_i$ in local coordinates. We take the symplectic form to be $\omega=d\lambda$.

Throughout the text, we take our convention $\iota_{X_H}\omega = dH$. This is the opposite of many conventions in symplectic geometry (for example, in Abouzaid's paper, \textit{Symplectic Cohomology and Viterbo's Theorem} \cite{abouzaid_viterbo}) for which we apologize. In the case of symplectic cohomology, a $\tilde{J}$-holomorphic section $\tilde{u}:\mathbb{R}\times S^1\to\mathbb{R}\times S^1\times M, (s, t)\mapsto (s, t, u(s, t))$ with a fiberwise symplectic form $\omega+dH\wedge dt$ and a compatible almost complex structure $\tilde{J}$ induced from an almost complex structure $J$ on $M$, is equivalent to a solution of the Floer equation $\partial_s u + J(\partial_t u - X_H) = 0$. 

We normalize our conventions against the conventions in \cite{abouzaid_viterbo}: our differential points are in the opposite way as \cite{abouzaid_viterbo}. In \cite{abouzaid_viterbo}, $\langle dx_1,x_0\rangle$ counts $J$-holomorphic cylinders with its $+\infty$ end at $x_1$ and $-\infty$ end at $x_0$. In our conventions for the differential for $\langle dx_1,x_0\rangle$, we count $J$-holomorphic cylinders with $+\infty$ end at $x_0$ and $-\infty$ end at $x_1$. Because of these canceling conventions\footnote{If $u$ satisfies $\partial_su+J(\partial_t u-X_H) = 0$ then $v(s, t) := u(-s, -t)$ satisfies $\partial_sv+J(\partial_t v+X_H) = 0$.}, when we say symplectic cohomology in our setting (as well as taking products) it agrees with Abouzaid's conventions in \cite{abouzaid_viterbo}.

The main source of the above discrepancy is that for the purposes of defining fixed point Floer cohomology, we follow the SFT convention for the differential, for which, if we wish to compute Floer homology, we count $ \langle dx_1,x_0 \rangle$ with $x_1$ at $s=+\infty$, and $x_0$ at $s=-\infty$, and the opposite for cohomology. This is the convention we followed throughout this sequence of papers \cite{JYZ, yao-zhao}. We decided to consider fixed point Floer cohomology for positive Dehn twists because under the above conventions the product operation on the fixed point Floer cohomology is particularly rich, seeing a finite sector of the Chas-Sullivan product in string topology. 

Generally in the Floer homology community (for instance \cite{abouzaid_viterbo}, and \cite{ganatra_thesis}) the convention is that homology differential counts $\langle dx_1,x_0 \rangle$ with $x_1$ at $s=-\infty$, and $x_0$ at $s=+\infty$, and the opposite for cohomology. When we encounter Floer homology groups defined using this convention, we will use the subscript $FL$ to indicate this. For example, if $\phi$ is a symplectomorphism on a closed symplectic manifold $M$, we have

 \[
 HF^*(M,\phi) \cong HF^*_{FL}(M,\phi^{-1})
 \]
 
 and the pair of pants product 
 \[
 HF^*(M,\phi)\otimes HF^*(M,\phi) \rightarrow HF^*(M,\phi^2)
 \]
 is identified with (on a set theoretic level) 
 \[
 HF^*_{FL}(M,\phi^{-1})\otimes HF^*_{FL}(M,\phi^{-1}) \rightarrow HF^*_{FL}(M,\phi^{-2}).
 \]

% \textit{Comparison:}

% Abouzaid Viterbo paper: $J \partial_s u = (\partial_t u - X)$ and $x_1$ (input) at $+\infty$, and $x_0$ (output) at $-\infty$ (p. 14), with $\omega(X, \cdot) = -d H$.

% Sheel's thesis: $\partial_s u = - J(\partial_t u - X)$ and $x_1$ (input) at $+\infty$ and $x_0$ (output) at $-\infty$ (p. 27)

\subsubsection*{Size of constants}
In this section we explain our choices of constants we will use throughout this paper. For the differential we shall always fix $n$, which is the number of times we are Dehn twisting. For the product operation we will always fix $n$ and $m$, which are the iterates of Dehn twists we are putting into the product.

We shall fix a large number $C>0$. We shall take two constants $\epsilon>0$ and $\epsilon_G>0$ which, after fixing $C$, will be sufficiently small. How small they need to be will be specified in the proof. After taking $\epsilon, \epsilon_G$ to be sufficiently small, we will introduce another parameter $\delta_*$ which is taken to be small relative to the fixed parameters $\epsilon_G$ and $\epsilon$. This parameter $\delta_*$ will control the size of the perturbations when we perturb Morse-Bott orbits to be Morse.

For confinement of the product, the general logic will be that we choose $C, \epsilon, \epsilon_G$ to be small enough to confine all the $J$-holomorphic curves in the Morse-Bott setting: then we will take another small perturbation of size $\delta$ so as to break the Morse-Bott degeneracy.

\subsection*{Acknowledgments}

We would like to thank Ivan Smith, Yankı Lekili, Kyler Siegel, Bernd Siebert and Jean Gutt.

MJ was supported via UKRI Horizon Europe grant FloerPlus35 (EP/X030660/1). YY was partially suported by ERC Starting Grant No. 851701 and ANR COSY ANR-21-CE40-0002.
\clearpage

\section{Geometric setup}

\subsection{Definition of the Dehn twist}

We begin by describing the Dehn twist around a Lagrangian sphere in all dimensions $\geq 4$.

We start with an Liouville domain $(M^{2n}, d\theta_M)$ in dimension $2n\ge 4$ with boundary $\partial M=P$. Let $L$ be a Lagrangian sphere $S^n\rightarrow L \subset M$. In what follows, we shall describe the definition of a Dehn twist $\phi: M\rightarrow M$ around $L$, which is a symplectomorphism of $M$ supported in a tubular neighborhood of $L$.

By the Weinstein neighborhood theorem, there exists an embedding $\iota: (D^*S^n,d\theta_{T^*S^n}) \rightarrow (M^{2n}, d\theta_M)$ such that $\iota^* d\theta_M = d\theta_{T^*S^n}$ and $\iota$ restricted to the zero section is the embedding $S^n \to L \subset M$ above.\footnote{We take the convention that in local coordinates $\theta_{T^*S^n}=\sum p_jdq_j$, where $p_j$ is the fiber coordinate.} We then modify the Liouville form $\theta_M$ locally\footnote{Fixed point Floer cohomology only depends on the symplectic form $d\theta_M$ and not on the primitive. The Dehn twist operation may depend, however, on the choice of framing $\iota: (D^*S^n,\theta_{T^*S^n}) \rightarrow M$. We suppress this from the notation.} so that in fact
\[
\iota^* \theta_M = \theta_{T^*S^n}.
\]
To be more precise, we have $\iota^*d\theta_M = d\theta_{T^*S^n}$ by the Weinstein neighborhood theorem, so it follows that there is some function $f$ on $D^{\ast} S^n$ such that
\[
\theta_{T^*S^n}-\iota^*\theta_M = df
\]
Here we are using the fact that $H^1(S^n,\mathbb{R})=0$ when $n \geq 2$. Then we take $\theta_M' = \theta_M+ d(\beta f)$ where $\beta$ is an appropriate cutoff function.

Now that we have an exact embedding from $(D^*S^n,\theta_{T^*S^n})$ to $M$, it suffices to construct a compactly supported symplectomorphism on $(D^*S^n,\theta_{T^*S^n})$ and extend it to rest of $M$ (and since we are doing Floer theory we eventually need to add small perturbations to make it nondegenerate) which is what we turn to next.

\subsubsection{Dehn twists on $T^*S^n$}

We first review how to define the Dehn twist (and its iterates) as a compactly supported symplectomorphism on $T^*S^n$. We follow the treatment in \cite{seidel_long_exact_sequence}.

We begin by considering the Lefschetz fibration with total space $\mathbb{C}^{n+1}$ and projection 
\[
q: \mathbb{C}^{n} \rightarrow \mathbb{C}_t
\]
given by 
\[
t= q(Z_1,Z_2...,Z_{n+1})=Z_1^2+Z_2^2+\dots +Z_{n+1}^2. 
\]
Here we take $(Z_1,...,Z_{n+1})$ to be complex coordinates on $\mathbb{C}^{n+1}$ and we shall use $Z_j = X_j + iY_j$ to denote their real and complex components.

Our description of $T^*S^n$ will always be given by
\[
T^*S^n= \{\xi+i\eta \in \mathbb{C}^{n+1} \; : \; |\xi|^2=1, \langle \xi,\eta\rangle =0\}
\]
where $(\xi,\eta)$ are vectors in $\mathbb{R}^{n+1}$, and $\langle -,-\rangle$ is the standard real inner product. We will also use the induced metric to define the fiber length $\|\eta\|$.\footnote{It follows that the Reeb orbits on the unit co-sphere bundle $S^*S^n$ are $2\pi$-periodic.} Now we would like to identify the above description of $T^*S^n$ with $q^{-1}(1)$.

We first describe $q^{-1}(1)= \left\{ \sum_{i=1}^{n+1}Z_i^2=1 \right\}$ as the set
\[
q^{-1}(1) = \left\{(X_1,X_2,...,X_{n+1}) +i(Y_1,Y_2,...,Y_{n+1}) \; \bigg| \; \sum_{j=1}^{n+1}X_j^2 -\sum_{j=1}^{n+1}Y_j^2=1, \quad \langle X,Y \rangle =0\right\}
\]
Here $X=(X_1,X_2,...,X_{n+1})$ and $Y= (Y_1,Y_2,...,Y_{n+1})$.

Then we may identify $T^*S^n$ with $q^{-1}(1)$ as follows:
if we let $\phi_1: q^{-1}(1)\rightarrow T^*S^n$ denote the map
\[
Z=X+iY\rightarrow ( X\|X\|^{-1},-Y\|X\|) = (\xi, \eta)
\]
where the first factor $\xi$ is the $S^n$ direction and the second factor $\eta$ is the fiber direction.  This is a diffeomorphism from the fiber $q^{-1}(1)$ to $T^*S^n$.

With the above understood, we next describe the Dehn twist on $T^*S^n$ in more concrete terms. To be more precise, we describe the Dehn twist on $T^*S^n$ as the image of a Hamiltonian diffeomorphism away from the zero section (and this Hamiltonian diffeomorphism extends smoothly to the zero section as the antipodal map). To do this, we shall use the identification between $T^*S^n$ and $q^{-1}(1)$. We define
\[
\Sigma_c = \sqrt{c}S^n  = \{\sqrt{c} (X_1,0,X_2,0, \cdots, X_{n+1}, 0) \in \mathbb{R}^{n+1} \subset \mathbb{C}^{n+1} \; | \; X_1^2+X_2^2+\cdots X_{n+1}^2 =1 \} \subset q^{-1}(c)
\]
and define $\Sigma$ to be the union of all $\Sigma_c$'s (including $c=0$). We will also denote by $\Sigma^*$ the union of $\Sigma_c$'s for $c\ne 0$. 

We will now describe the Dehn twist. Let $Z = (Z_1, Z_2, \cdots, Z_{n+1})\in \mathbb{C}^{n+1}$. We let $\hat{Z} = e^{-i\alpha/2}(Z_1, Z_2, \cdots, Z_{n+1})$, where $\alpha:= \arg(q(Z))$.
In Seidel's notation, consider the map
\[
\Phi: \mathbb{C}^{n+1} \setminus \Sigma \rightarrow \mathbb{C} \times (T^*S^n\setminus T(0))
\]
defined by\footnote{Here the choice of $\alpha/2$ does not matter; the two factors of $-1$ cancel in the right hand side of $\Phi$.}

\[
\Phi(Z_1, Z_2, \cdots, Z_{n+1}) = (q(Z_1, Z_2, \cdots, Z_{n+1}), \sigma_{\alpha/2}(re(\hat{Z})\|re(\hat{Z})\|^{-1},-im(\hat{Z})\|re(\hat{Z})\|))
\]

Here $T(0)$ means the zero section, and $\sigma_t(\eta,\xi) :=(\cos(t)\xi+\sin(t) \frac{\eta}{\|\eta\|},\cos(t)\eta-\sin(t)\|\eta\|\xi )$.

Then $\Phi$ is a diffeomorphism. We also have
\[
(\Phi^{-1})^* \theta_{\mathbb{C}^{n+1}} = \theta_{T^*S^n} - \tilde{R}_s(\mu)d\alpha, \quad \tilde{R}_s(t) = \frac{1}{2}t - \frac{1}{2}(t^2+s^2/4)^{1/2}
\]
where $\theta_{\mathbb{C}^{n+1}}$ is the standard Liouville form on $\mathbb{C}^{n+1}$ given by
\[
\frac{1}{2}(X_1dY_1-Y_1dX_1+...+X_{n+1}dY_{n+1}-Y_{n+1}dX_{n+1}),
\]
$\theta_{T^*S^n}$ is the Liouville form on $T^*S^n$,
$\alpha$ is the angular coordinate of $q(Z)$, and $\mu$ is the absolute value of the fiber direction in $T^*S^n$ (in our above conventions, $\mu = |\eta|$). We also have $s=|q(Z)|$. Now for each $s>0$, consider the the Hamiltonian flow of $\tilde{R}_s(\mu(y))$, which we can describe explicitly as
\[
\phi(y) = \sigma_{2\pi \tilde{R}_s'(\mu(y))}(y)
\]
away from the zero section $T(0)$. This extends smoothly over the zero section $T(0)$: see \cite[Lemma 1.8]{seidel_long_exact_sequence}. However, it is not compactly supported. If we make a slight modification of $\tilde{R}_1$ to make it compactly supported away from the zero section, then we obtain a compactly supported symplectomorphism of $T^*S^n$. We shall call this \textit{Seidel's model Dehn twist}.

For our purposes, we need a slightly different modification of the Dehn twist for $\mu$ near $1$.

\begin{definition}
Let the functions $G_1(t)$ be defined as follows: for $t<1/4, G_1(t)=\tilde{R}_1(t)$.
For $t$ near $t=1$, $G_1(t)= -\frac{1}{2}\epsilon_G(t-1)^2$. Between these regions, $-G_1(t)$ is smooth, and its derivative with respect to $t$ is monotonically increasing.\footnote{The number $\epsilon_G$ is chosen so that $G$ is sufficiently $C^2$ small for $t$ in a neighborhood of $1$. We will more carefully describe how small $\epsilon_G$ needs to be when we introduce fixed point Floer cohomology.}
\end{definition}

We will use the functions $G_1(t)$ above instead of $\tilde{R}_s(t)$ to describe our Dehn twists. We shall call the pair
\[
(T^*S^n|_{\mu\leq 1+\epsilon}, \phi_{G_1})
\]
where $\phi_{G_1}$ is the Dehn twist defined as above using the function $G_1(t)$ \emph{adapted model Dehn twists}.

Our computation of fixed point Floer cohomology (and its product) is performed by counting $J$-holomorphic curves in the mapping torus. Using the above, we will now rephrase the Dehn twist as the existence of a certain stable Hamiltonian structure on a mapping torus.

As a first approximation, the total space of the mapping torus can be taken to be $\mathbb{C}^{n+1}|_{s=1}$, with the 1-form $q^*d\alpha$ and the two-form being $d\theta_{\mathbb{C}^{n+1}}$ modified slightly away from $\Sigma$ (which is the union of the zero sections of $T^*S^n$). 

To get to the precise definition,
we shall consider the mapping torus of the codisk bundle $D^*S^n$ (instead of $T^*S^n$) and we shall view it as being built out of two pieces glued together. The first piece is $S^1 \times (T^*S^n\setminus T(0))$ together with the one-form $d\alpha$ and the two-form $d (\theta_{T^*S^n} - G_1(\mu)d\alpha)$, and the second piece is $\overline{\Phi^{-1}((S^1\times D_{\le \epsilon}^*S^n)\setminus T(0))}\subset \mathbb{C}^{n+1}$, together with ($q^*d\alpha$,  $d\theta_{\mathbb{C}^{n+1}}$). Taking the closure ensures we include the zero section $\Sigma_c$ for all $c\in S^1$. The two pieces are identified via the map $\Phi$.

Let us denote by $Y_{\phi_{G}}$ the mapping torus of $\phi_{G_1}$ with domain $D^*_{\mu\leq 1+\epsilon}S^n$ (which restricts to the antipodal map on the zero section). The standard symplectic form $d\theta_{T^*S^n}$ induces a closed two-form $\omega_Y$ on the mapping torus $Y_\tau$. The stable Hamiltonian structure is given by $(d\alpha, \omega_Y)$.
In the above, the first piece can be viewed naturally as a subset of $Y_{\phi_{G}}$. The identification is given by $(\alpha, p)\mapsto [(\alpha, \phi_{-\tilde{G}_1}^{\alpha/2\pi}(p))]$, where $\phi_{-{G}_1}^{\alpha}$ is the time-$\alpha$ map\footnote{In this paper, the Hamiltonian vector field $X_H$ of a Hamiltonian $H$ is defined as $\iota_{X_H}\omega = dH$.} of the Hamiltonian $-{G}_1$. The second piece is naturally included as a subset of $Y_{\phi_G}$.

\subsubsection{Iterations of Dehn twists on $T^{\ast} S^n$}

We now write down the stable Hamiltonian structures relevant for the $k$-iterations of Dehn twists: for the case of adapted Dehn twists we shall denote them by $\phi_{G}^k$.

We let $E^k=\{(t,Z_1, Z_2, \cdots, Z_{n+1})\in \mathbb{C}_t^*\times \mathbb{C}^{n+1} \; | \; t^k= Z_1^2+Z_2^2+\cdots+Z_{n+1}^2\}$, and $E_0^k:= E^k \setminus \Sigma$, where $\Sigma$ is given by
\begin{equation*}
    \Sigma = \{(t,\sqrt{t^k}(X_1, 0, X_2, 0, \cdots, X_{n+1}, 0)) \in \mathbb{C}_t^*\times \mathbb{C}^{n+1} \; | \; t\ne 0 \; \text{and} \; (X_1, X_2,\cdots, X_{n+1})\in S^n\subset \mathbb{R}^{n+1} \}
\end{equation*}

Then $E^k$ has a projection map $q:E^k\rightarrow \mathbb{C}_t^*$, which is the projection to the $\mathbb{C}_t$ component. We put the following fiberwise Liouville form $\theta_0$ on $E^k$: the canonical Liouville form defined on $\mathbb{C}^{n+1}$ restricted to $E^k$.

In this language, we consider $\Phi_0:E_0^k \rightarrow \mathbb{C} \times (T^*S^{n}\setminus T(0))$
defined by

\[
\Phi_0(t,Z_1, Z_2, \cdots, Z_{n+1}) = (t^k, \sigma_{k\alpha/2}(re(\hat{Z})\|re(\hat{Z})\|^{-1},-im(\hat{Z})\|re(\hat{Z})\|))
\]
which satisfies 
\[
(\Phi_0^{-1})^* \theta_0 = \theta_{T^*S^n} - k\tilde{R}_{s^k}(\mu)d\alpha
\]

near the zero section, where we use the polar coordinate $t=s e^{i\alpha}$, and $$\hat{Z} = e^{-ik\alpha/2}(Z_1, Z_2, \cdots, Z_{n+1}).$$ 

We define the total space of the bundle $B^k$ to be $E^k|_{|t|=1}$, which can be viewed as the pullback of $B^1$ under the $k$-th covering map $t\mapsto t^k$. The stable Hamiltonian structure of the bundle is given by $d\alpha$ and the restriction of $d\theta_0$, which, under the identification $\Phi$, is the pair $(d\alpha, d (\theta_{T^*S^2} - k\tilde{R}_{s^k}(\mu)d\alpha))$ on $S^1\times (T^*S^n\setminus T(0))$.

With the above we recover \emph{Seidel's model for iterates of a Dehn twist}. As before we may modify $\widetilde{R}_{1}(\mu)$ to $G_1(\mu)$ so that the stable Hamiltonian structure away from the zero section is given by $(d\alpha,d (\theta_{T^*S^2} - kG_{1}(\mu)d\alpha))$ to obtain the \emph{adapted model for iterates of a Dehn twist}. We will use the notation $B^k_{\mu\leq 1+\epsilon}$ to denote the mapping torus of the adapted model for the $k$-th iterate of Dehn twists with the associated stable Hamiltonian structure (note that we only specified what $G_1(\mu)$ is for $\mu \leq 1+\epsilon$ ).

\begin{remark}\label{rmk_hamiltonian}
For even iterations of Dehn twists, it will sometimes be convenient to describe the stable Hamiltonian structure in the following way.  Let $k$ be even. We break the manifold $B^k_{\mu \leq 1+\epsilon}$ into two pieces.
On the first piece, the stable Hamiltonian structure on $S^1 \times (D^*_{\mu\leq 1+\epsilon}S^n\setminus T(0))$ is given by $(d\alpha, d (\theta_{T^*S^n} - kG_1(\mu)d\alpha))$. On the second piece, we consider the stable Hamiltonian structure on $S^1 \times D^*_{\mu\leq \epsilon}S^n$ given by $(d\alpha, d(\theta_{T^*S^n} +H(\mu)d\alpha))$ where 
\[
H(\mu) = \frac{k\mu}{2} - kG_1(\mu)
\]
which has zero slope at $\mu=0$. The second piece is identified with the first piece via the gluing map
\[
(\mu, y, \alpha)\mapsto (\mu, \psi^{k\pi}(y), \alpha)
\]
where $\psi^{k\pi}$ is the time-$k\pi$ map of the Reeb flow on $S^*S^n$ (the contact form being the restriction of $\theta_{T^*S^n}$). It follows that the gluing map pulls \[\theta_{T^*S^n} - kG_1(\mu)d\alpha = \mu \lambda_{S^*S^n} - kG_1(\mu)d\alpha\] back to 
\[
\mu (\lambda_{S^*S^n}+\frac{k}{2}d\alpha) - kG_1(\mu)d\alpha = \mu \lambda_{S^*S^n} +H(\mu)d\alpha.
\]

\subsection{Iterated Dehn twists on $M$} \label{sec:iterated_Dehn_twist_on_M}

We now extend the above (adapted) model Dehn twists to the entire manifold $M$. To be concrete, we will describe a stable Hamiltonian structure on the mapping torus of such Dehn twists. We fix $k \in \mathbb{Z}_{>0}$ as the number of iterations.

Recall we have an inclusion $\iota: (D^*S^n,\theta_{T^*S^n})\rightarrow (M^{2n},\theta_M)$ so that $\iota^*\theta_M = \theta_{T^*S^n}$. For any $k\in \mathbb{Z}_+$, we choose $f_1:M^{2n}\setminus \iota(D^*_{\mu\leq 1-\epsilon}S^n)\rightarrow \mathbb{R}$ so that near the boundary of $\iota(D^*S^n)$, $kf_1$ agrees with $-kG_1(\mu)$ and $kf_1$ extends to a $C^2$ small Morse function in the rest of $M^{2n}\setminus \iota(D^*S^n)$.

Given $\epsilon>0$,
we will now construct a stable Hamiltonian structure on $S^1\times (M^{2n}\setminus \iota(D^*
_{\mu\leq 1-\epsilon}S^n) )$ which we will informally refer to as the \textit{Morse region}. We use $\alpha$ to denote the $S^1$ coordinate, then the stable Hamiltonian structure is simply given by
\[
(d\alpha, d(\theta_M + kf_1d\alpha)).
\]
The Reeb flow of this stable Hamiltonian structure recovers the gradient flow of $f$ (with an appropriate choice of sign convention) over $S^1\times (M^{2n}\setminus \iota(D^*_{\le 1-\epsilon}S^n))$.

Next, we glue the stable Hamiltonian structure over $S^1\times (M^{2n}\setminus \iota(D^*_{\mu\leq 1-\epsilon}S^n) )$ to a subset of the stable Hamiltonian structure on $B^k$.

To be precise, on $B^k$ we use $B^k_{\mu\leq 1+\epsilon}$ to mean that for each fiber of $B^k$ (which is identified with $T^*S^{n}$) we only include covectors of length $\leq 1+\epsilon$, and further use $B^{k}_{0<\mu\leq 1+\epsilon}$ to indicate that we remove the zero section. We will refer to $B^k_{\mu\leq 1+\epsilon}$ as the \textit{Dehn twist region}.

We use $\Phi_0$ to identify  $B^k_{0<\mu\leq 1+\epsilon}$ with the set
\[
S^1 \times (T^*S^n\setminus T(0))
\]
with the stable Hamiltonian structure $(d\alpha, d(\theta_{T^*S^n} -kG_1(\mu)d\alpha))$.

The map $\iota$ identifies the boundary of the Morse region and the Dehn twist region, and the two stable Hamiltonian structures glue to a single stable Hamiltonian structure over this identification.

We shall use the notation $SHS(M^{2n}, kG_1,kf_1)$ to denote the above data of both the underlying manifold, i.e., the mapping torus of the Dehn twist, as well as the stable Hamiltonian structure we described above. We will denote the mapping torus with the associated stable Hamiltonian structure, viewed as a bundle over $S^1$, by $\mathcal{E}^k$. We note that by construction, there is a global 1-form on $\mathcal{E}^k$, equal to $\theta_{T^*S^n}-kG_1(\mu)d\alpha$ on $B^k_{0<\mu\le 1+\epsilon}$ (extending smoothly to the zero section) and equal to $\theta_M +kf_1d\alpha$ in the Morse region. We denote this global 1-form by $\alpha^k$.  This is essentially the stable Hamiltonian structure that we ultimately want to work with, except we shall later make small perturbations to $\alpha^k$ so that all fixed points are nondegenerate.

\begin{remark}\label{rmk:mapping_torus_charts}
What we have described can be phrased in the following way. There is a symplectomorphism $\tau_{kG_1,f_k}$ that acts as the Dehn twist $\tau_{kG_1}$ in $\iota(D_{\mu\leq 1}S^n)$ and acts as the Hamiltonian flow of $kf_1$ on the complement. Let $Y_{\tau_{kG_1,f_k}}$ denote the mapping torus with the naturally associated stable Hamiltonian structure. We can cover this mapping torus (viewed as a stable Hamiltonian manifold) using two charts, one chart for the Dehn twist region as
\[
(B^k_{\mu\leq 1+\epsilon}, (\alpha, d\alpha^k))
\]
and the other, the Morse region, as
\[
(S^1\times (M^{2n}\setminus \iota(D^*
_{\mu\leq 1-\epsilon}S^n)), (\alpha, d\alpha^k)).
\]
\end{remark}

\subsection{Cobordisms of Dehn twists}\label{sec:cobordism_of_Dehn_twists}

For the purpose of describing product operations on fixed point Floer cohomology of Dehn twists, we need to construct a certain symplectic cobordism between mapping tori of Dehn twists with the properties we shall describe below.

Let $B_3 = S^2 \setminus \{0,1, \infty\}$ denote the thrice-punctured Riemann sphere, equipped with the following data: a choice of negative cylindrical coordinates $(s_0,t_0)\in (-\infty,0]\times S^1$ at $z=0$, $(s_1,t_1)\in (-\infty,0] \times S^1$ at  $z=1$, and a choice of positive cylindrical coordinates $(s_\infty,t_\infty)\in [0,\infty) \times S^1$ at $z = \infty$.
For every pair of positive integers $m, n$, we choose a smooth map $g_{m, n}:B_3\to S^1$ such that near the three punctures, $g_{m,n}$ is described by $t_1\mapsto mt_1$, $t_2\mapsto nt_2$ and $t_\infty \mapsto (m+n)t_\infty$. We will denote by $E^{m, n}$ the pullback bundle $g_{m, n}^*\mathcal{E}^1$.

Notice that $E^{m,n}$ is a locally Hamiltonian fibration over $B_3$ with the following properties:

\begin{enumerate}
\item We equip the base $B_3$ with the symplectic form $\omega_{B_3}$ that takes the form $Kds_i\wedge dt_i$ for coordinates $(s_i,t_i)$ near any of the punctures, where $K$ is a sufficiently large constant (which depends on $m$ and $n$). Then, by adding the symplectic form above to the vertical 2-form, $E^{m,n}$ becomes a symplectic manifold.  In a neighborhood of the the punctures $z=0,1,\infty$, the bundle $E^{m,n}$ is the symplectization of the mapping torus with Dehn twists of order $m,n$ and $m+n$ respectively. For instance, near $z=\infty$, the bundle $E^{m,n}$ takes the form
\[
[0,\infty)\times \mathcal{E}^{m+n}
\] 
with the symplectic form given by $Kds_\infty \wedge dt_\infty + d\alpha ^{m+n}$. In these coordinates, the stable Hamiltonian structure on $\{0\}\times \mathcal{E}^{m+n}$ is the pair $(dt_\infty, d\alpha^{m+n})$.

\item There is a globally defined $1$-form $\lambda^{m,n} = g_{m, n}^*\alpha^1$ on $E^{m,n}$, such that $d\lambda^{m,n}$ is the vertical 2-form of the locally Hamiltonian fibration. Moreover, over the punctures, by viewing $E^{m,n}$ as the symplectization of the stable Hamiltonian structure, $\lambda^{m,n}$ agrees with primitive of the vertical two form constructed in the previous section. For example near the $z=\infty$ puncture, after identifying $E^{m,n}$ with $[0,\infty)\times \mathcal{E}^{m+n}$ as above, we have $\lambda ^{m,n} = \alpha^{m+n}$. A similar relationship holds in the other punctures.

\item Let $\pi^{m,n}$ denote the projection $E^{m,n}\rightarrow B_3$, then the symplectic form on $E^{m,n}$ is $\pi^{m,n *} \omega_{B_3} + d\lambda^{m,n}$.
\end{enumerate}

Just as in the case of $\mathcal{E}^k$, we decompose the symplectic manifold $E^{m,n}$ into two pieces glued along a common boundary.

Note that $\mathcal{E}^1$ decomposes as $B^1_{\mu\leq 1+\epsilon} \cup (S^1\times (M\setminus \iota D_{\mu \leq 1-\epsilon}^*S^n))$. Then the first piece of $E^{m,n}$ is $g_{m,n}^*B^1_{\mu\leq 1+\epsilon}$, with the vertical 2-form being the restriction of $d\lambda^{m,n}$. We shall call this piece the \textit{Dehn twist region}. Note that in this Dehn twist region, using the map $\Phi$, we can identify a subset of the Dehn twist region given by $g_{m,n}^*B^1_{0<\mu\leq 1+\epsilon}$. Topologically, this is $B_3 \times (D^*S^n\setminus T(0))$, where $\lambda^{m,n}$ restricts to $
\theta_{T^*S^n}-G_1(\mu) dg_{m,n}$.

The second piece is simply $g_{m,n}^*(S^1\times (M\setminus \iota D_{\mu\leq 1-\epsilon}^*S^n))$. Topologically, this is just
\[
B_3 \times (M\setminus \iota D^*_{\mu \leq 1-\epsilon}S^n)
\]
where $\lambda^{m,n}$ restricts to $\theta_M +f_1 dg_{m,n}$. We shall call this second piece the \textit{Morse region} of $E^{m,n}$.

\end{remark}

\subsection{Review of fixed point Floer cohomology}

We shall first give a general review of fixed point Floer cohomology and its associated structures, then we shall describe how to adapt it to the setting of the stable Hamiltonian structures we have constructed above.

Recall that we let $(M,d\theta_M)$ be a Liouville domain. Let $\phi: M\to M$ be a  symplectomorphism supported away from the boundary. The boundary $\partial M$ of $M$ is contact and we fix a collar neighborhood of $\partial M$ as in Remark \ref{rmk:H_near_boundary}.
We make a small Hamiltonian perturbation so that 
$\phi$ is nondegenerate, that is, for every fixed point $x$ of $\phi$, the linearization $d\phi_x$ does not have $1$ as an eigenvalue. The perturbation should take the form as specified in Remark \ref{rmk:H_near_boundary} near the boundary so that there are no fixed points near the boundary.

The fixed point Floer cohomology is the cohomology of the cochain complex ($CF^*(M, \phi)$, $d$), whose underlying module is generated over\footnote{We may also use $\mathbb{C}$ or $\mathbb{Z}$ coefficients, but for this quick introduction we stick with $\ZZ_2$ coefficients for simplicity.} $\ZZ_2$ by all fixed points of $\phi$. 

The differential of $CF^*(M, \phi)$ is defined by counting $J$-holomorphic cylinders in the symplectization of the mapping torus of $\phi$. More precisely, for any symplectomorphism $\phi:M\to M$, the mapping torus $Y_\phi$ is defined as
\begin{equation}
    Y_\phi = [0,1]_t\times M/((1,p)\sim(0,\phi(p))).
\end{equation}
The mapping torus comes with a projection $\pi: Y_\phi\to S^1_t$, and the symplectic form $d\theta_M$ induces a closed 2-form $\omega_\phi\in\Omega^2(Y_\phi)$ which restricts to $d\theta_M$ on each fiber. The vector field $\partial_t$ on $[0,1]_t\times \Sigma$ descends to a vector field on $Y_\phi$, which we still denote by $\partial_t$. Notice that there is a one-to-one correspondence between fixed points of $\phi$ and closed orbits of $\partial_t$ that cover $S^1_t$ once. We will denote by $\gamma_x$ the closed orbit associated to a fixed point $x$.

The projection $\pi$ now extends to $\mathbb{R}\times Y_\phi\to \mathbb{R}\times S^1$, and the fiberwise symplectic form $\omega_\phi$ extends to the symplectization as well. The symplectization of $Y_\phi$ is the manifold $\mathbb{R}_s \times Y_\phi$ together with the symplectic form $ds\wedge dt+\omega_\phi$. Choose a compatible almost complex structure (see Definition \ref{defn:boundary_requirement_J}: this choice of $J$ also prevents $J$-holomorphic curves from escaping through the boundary) on the mapping torus that is regular\footnote{This means all moduli spaces of $J$-holomorphic cylinders we care about are transversely cut out: this holds for a generic choice of $J$.}. Given two fixed points $x$, $y$ of $\phi$, we define the moduli space $\mathcal{M}^J_{x,y}$ to be
\begin{equation}
\mathcal{M}^J_{x,y}:=\{ u: \mathbb{R}_s \times S^1_t \to \mathbb{R}\times Y_\phi \; | \; \pi \circ u = id; \; \partial_s u+J\partial_t u=0; \lim_{s\to \infty} u(s,\cdot) =\gamma_x,\lim_{s\to -\infty} u(s,\cdot) =\gamma_y \}.
\end{equation}

Now if $\phi$ is exact (which will be the case for Dehn twists), the differential\footnote{Notice the differential goes the \emph{opposite} way as our previous paper, since we are now describing \emph{cohomology}.} on $CF^*(M, \phi)$ is defined by
\begin{equation}\label{count2}
    \langle d x,y \rangle := \#_{\ZZ_2} (\mathcal{M}_{y,x}/\mathbb{R}).
\end{equation}

Under the exactness assumption, for a regular almost complex structure $J$ on $\RR \times Y_{\phi}$, the set of Fredholm index one $J$-holomorphic sections (modulo the natural $\mathbb{R}$-action by translation) is a compact $0$-dimensional manifold, and $\#_{\ZZ_2} (\mathcal{M}^J_{y,x}/\mathbb{R})$ denotes the mod-2 count of points in the moduli space. If $J$ is regular then also $d^2=0$, and we will denote by $HF^*(M, \phi)$ the cohomology of the complex $(CF^*(M, \phi),d)$. The cohomology $HF^*(M,\phi)$ is invariant under (compactly supported) Hamiltonian isotopies of $\phi$ and independent of choice of (regular) $J$, see for example \cite{seidel_thesis}. 

Suppose $\phi: M\rightarrow M$ is an exact symplectomorphism, and assume there is an exact symplectic fiber bundle $(X,\pi_X, \omega_X)$ over the thrice-punctured sphere $B_3$, which, near the three punctures, is symplectomorphic to $[0,-\infty)\times Y_{\phi^n}$, $[0,-\infty)\times Y_{\phi^m}$ and $(+\infty, 0]\times Y_{\phi^{n+m}}$ respectively. We have already given a  careful construction of this bundle for the case of Dehn twists in the preceding sections.

For now, let us observe that, such a bundle induces a map\footnote{Generally one has to use Novikov coefficients to deal with the possibility of there being infinitely many nontrivial $J$-holomorphic sections. However this is not really relevant for our purposes: see Section \ref{sec:prod_str_Dehn_twists} for details.}
\begin{equation}
    \bullet: HF^*(M, \phi^n)\otimes HF^*(M, \phi^m)\longrightarrow HF^*(M, \phi^{n+m})
\end{equation}
In fact, this actually works for any two symplectomorphisms $\psi, \phi$:
\begin{equation*}
    \bullet: HF^*(M, \psi)\otimes HF^*(M, \phi)\longrightarrow HF^*(M, \psi \circ \phi)
\end{equation*}
This is what we call the (pair-of-pants) \emph{product structure} on the fixed point Floer cohomology \footnote{This corresponds to the \textit{coproduct} structure for the fixed point Floer \textit{homology} in our earlier paper \cite{yao-zhao}.}.
It is defined by counting rigid pseudo-holomorphic sections of the bundle $X\to B_3$ with appropriate asymptotics. Namely, if $x$, $y$, $z$ are fixed points of $\phi$, $\psi$ and $\psi\circ\phi$ respectively, and $J$ is a regular tame almost complex structure (see the precise Definition \ref{defn:tame_for_cobordism}), then the moduli space $\mathcal{M}^J_{x,y;z}$ is defined by
\begin{equation}
    \mathcal{M}^J_{x,y;z} =\left\lbrace u: B_3\to X \;\middle|\;
  \begin{tabular}{@{}l@{}}
    $\pi_X\circ u$ is degree 1,\ $u$ is $J$-holomorphic, and\\
    $u$ is asymptotic to $\gamma_x$, $\gamma_y$ at its negative punctures and \\ $\gamma_z$ over its positive puncture.
   \end{tabular}
  \right\rbrace
\end{equation}
and when it has index zero this moduli space is a compact $0$-dimensional manifold.

The product on the chain level is now defined as:
\begin{equation}\label{cobordismcount}
    \langle x\bullet y, z\rangle = \#_{\ZZ_2}\mathcal{M}^J_{x,y;z}
\end{equation}
where $\#_{\ZZ_2}\mathcal{M}^J_{x,y;z}$ denotes the mod-$2$ count of Fredholm index $0$ sections.\footnote{Of course we can do $\mathbb{Z}$ counts if we choose orientations/use a system of coherent orientations as we discussed in \cite{yao-zhao} and \cite{JYZ}}

This `product' on cohomology is (graded) commutative and associative, and independent of choice of regular almost complex structure and (compactly supported) Hamiltonian isotopy classes of $\phi$.\footnote{That this product structure has these stated properties is a folklore result that follows from standard techniques: our results in the present work do not rely on them.}

\subsubsection{The case of Dehn twists: The differential}

We now give very concrete descriptions of the geometric data we need to compute fixed point Floer cohomology of Dehn twists and the associated product. We will heavily use the stable Hamiltonian structures constructed in previous sections.

Let $k\geq1$ and consider the bundle $\mathcal{E}^k \rightarrow S^1$ constructed in Section \ref{sec:iterated_Dehn_twist_on_M}, with the associated stable Hamiltonian structure $(\alpha, d\alpha^k)$. This is identified with the mapping torus in Remark \ref{rmk:mapping_torus_charts}. This is essentially already the stable Hamiltonian structure that we want, except that the Reeb orbits of this Hamiltonian flow are Morse-Bott. Here, we perform a small perturbation of the stable Hamiltonian structure to make the time-$1$ Reeb orbits nondegenerate. The perturbation and the associated Reeb orbits are described below.

We recall that $\mathcal{E}^k$ is split into the Morse region and the Dehn twist region, glued along a common boundary. Over the Morse region, the manifold is $S^1 \times (M\setminus \iota(D^*_{\mu<1-\epsilon}S^n))$ with the stable Hamiltonian structure $(\alpha, d(\theta_M +kf_1 d\alpha))$. We will later fix $0<\epsilon_G, \epsilon \ll 1 $ so that for $z\in M\setminus \iota(D^*_{\mu<1+\epsilon}S^n)$, $f_1$ is a $C^2$-small Morse function whose value is between $\frac{1}{2}\epsilon_G \epsilon^2$ and $\frac{3}{2} \epsilon_G\epsilon^2$.

Next, we examine what happens in the Dehn twist region. We split into two cases of $k$ even or odd.

\begin{enumerate}
\item In the case where $k$ is even, we can describe the Dehn twist region as $S^1\times D^*_{\mu\leq 1+\epsilon}S^n$ with stable Hamiltonian structure $(d\alpha, d\theta_{T^*S^n} +H(\mu)d\alpha)$, where $H(\mu) = k\mu/2-kG_1(\mu)$.

The Morse-Bott family of Reeb orbits can be described below:

\begin{itemize}
\item For $\mu=0$, we have an $S^n$ family of Reeb orbits
\item For each $0<\ell \leq k$, we find $\mu_{\ell}$ so that $H'(\mu_{\ell}) =\ell$. Then the full $\mu=\mu_{\ell}$ locus is a Morse-Bott family of Reeb orbits.
\end{itemize}

Fix $\delta\ll 1$: we make a size-$\delta$ Morse-Bott perturbation supported in $\mu \in [0,\delta]$ and $\mu \in [\mu_{\ell}-\delta, \mu_{\ell}+\delta]$ by modifying $G_1$ to $G_1 +\delta F$. Here $F$ depends on $\alpha$ (it is a time-dependent modification). 

To be more precise, on the $\mu=0$ locus, the $S^n$ family of fixed points can be perturbed to be nondegenerate by choosing a $C^2$-small perfect Morse function on $S^n$. For $\mu=\mu_{\ell}$, we can think of the perturbation as happening in two steps. First, we have that the $\mu=\mu_{\ell}$ locus is foliated by Reeb orbits of equal period, and we can view this as an $S^1$ bundle over a base manifold $Q_n$, where the $S^1$ fibers are the Reeb orbits. We choose a Morse function $f_Q$ on $Q_n$ and perturb by this Morse function (this is a time independent perturbation) so that the remaining Reeb orbits correspond to critical points of $f_Q$. Then each Reeb orbit corresponds to an $S^1$ family of Hamiltonian orbits. We choose a time dependent function $h(t)$ so that each Reeb orbit is perturbed to two distinct Hamiltonian orbits. This time-dependent perturbation corresponds to modifying $\alpha^k$ by $\delta h(\theta-\ell\alpha)d\alpha$ multiplied by some cut off function.

\item For the case of $k$ odd, the periodic orbits are all contained in the chart $S^1\times D^*_{\mu\leq 1+\epsilon}S^n\setminus T(0)$ with the stable Hamiltonian structure $(\alpha, d(\theta_{T^*S^n} -kG_1(\mu)d\alpha))$. The orbits do not appear at the zero section, they instead appear at $\mu_{\ell}$ for which $-kG_1'(\mu_{\ell})=\ell$. We use a similar perturbation method to achieve nondegeneracy at these orbits as well.

\end{enumerate}
With the above stable Hamiltonian structure chosen, we shall denote the Morse-Bott perturbed global 1-form as $\alpha^k_\delta$. We next specify which almost complex structures on $\mathbb{R} \times \mathcal{E}^k$ we want to use.
We begin with the observation that $\mathbb{R} \times \mathcal{E}^k$ are examples of \textit{locally Hamiltonian fibrations}, as defined in \cite{McDuff_Salamon}.

\begin{definition}
    The vertical distribution of a locally Hamiltonian fibration $\pi: E\to B$ with fiberwise symplectic form $\Omega$, is the sub-bundle $\text{ker } d\pi$ of $TE$. The horizontal distribution is the sub-bundle of $TE$ which is the $\Omega$-orthogonal complement of the vertical distribution.
\end{definition}

\begin{definition} \label{defn:boundary_requirement_J}
A compatible almost complex structure $J$ on $\mathbb{R}_s \times \mathcal{E}^k$ satisfies the following:
\begin{enumerate}
    \item $J$ is invariant under $\mathbb{R}$-translation;
    \item $J$ sends $\partial_s$ to the Reeb vector field of the stable Hamiltonian structure of $\mathcal{E}^k$; 
    \item $J$ preserves the vertical distribution $\Ver_k$ of the bundle $\mathcal{E}^k\to S^1$ and is compatible with the symplectic form on $\Ver_k$, and the projection to the base $\mathbb{R}\times S^1$ is holomorphic;
    \item $J$ is compatible with the symplectization structure on a collar neighborhood of the boundary of the fibers.
\end{enumerate}
\end{definition}

The special requirement (4) for the almost complex structure near the boundary of the fibers can be formulated precisely as follows: we observe that near the boundary the locally Hamiltonian fibration is trivial, which means we can write it as the product of a collar neighborhood of the boundary with the base. The boundary of $M$, $P=\partial M$ is a contact manifold, with contact form $\lambda_P$ and contact structure $\xi_P$. Since $M$ is a Liouville domain, there exists a collar neighborhood of $\partial M$ exact symplectomorphic to the symplectization $(\epsilon_b/2, \epsilon_b]_\mu\times P$ of the contact manifold $P = \partial M$. We choose an almost complex structure $J_0$ on $M$ such that near the boundary, $J_0$ is compatible with the described symplectization. In particular, $J_0$ preserves $\xi_P$ and sends $\partial_\mu$ to the Reeb vector field $R_P$ of $(P, \lambda_P)$. The almost complex structure $J$ is chosen to be a compatible almost complex structure that restricts to $J_0$ on the fiber.

\begin{remark} \label{rmk:H_near_boundary}
    In each connected component of the neighborhood $(\epsilon_b/2, \epsilon_b)_{\mu_i}\times P_i$ of the boundary of $M$ as described above, the Morse function $f_1$ takes the form $\theta_i\mu_i$, where $\theta_i$ is a small irrational number.  
\end{remark}

\begin{definition}\label{defn:confinement type acs}
    We say an almost complex structure $J_0$ on $M$ is of \textbf{confinement type}, if it satisfies requirement (4) near $\partial M$ as in Definition \ref{defn:boundary_requirement_J}, and has the following form in a neighborhood of $\iota (S^*S^n$): 
    
    View the neighborhood as $(1-\epsilon, 1+\epsilon)_\mu \times S^*S^n_y$: we require that in the region $((1-\epsilon, 1-\delta)\cup (1+\delta, 1+\epsilon))\times S^*S^n$, the almost complex structure $J_0$ is compatible with this symplectization\footnote{We restrict to these intervals because $(1-\delta,1+\delta)$ is the support of a perturbation that breaks the Morse-Bott degeneracy.}: that is, $J_0$ preserves the contact structure of $S^*S^n$ (the contact form is just the restriction of the Liouville form on $M$); is invariant in the $\mu$ direction on $((1-\epsilon, 1-\delta)\cup (1+\delta, 1+\epsilon)) \times S^{\ast} S^n$; and sends $\partial_\mu$ to the Reeb vector field on $S^*S^n$. 
    
    We say a compatible almost complex structure $J$ on $\mathbb{R}\times \mathcal{E}^k$ is of \textbf{confinement type} if it is the (unique) lift of a confinement-type almost complex structure $J_0$ on $M$.
\end{definition}

In the above we have implicitly used the following: we note that $\mathbb{R}_s\times \mathcal{E}^k \rightarrow \mathbb{R}_s \times S^1$ is a locally Hamiltonian fibration, where we fix the standard complex structure on the base $\mathbb{R}\times S^1$. We can uniquely specify an almost complex structure on $J$ on $\mathbb{R}_s\times \mathcal{E}^k$ by specifying an $s$-independent family of almost complex structures on the vertical bundle. To be more explicit, let $(s,t) \in [-1,1]^2$ denote local coordinates on the base: then in local coordinates the  fibration looks like the projection
\[
[-1,1]^2\times M \rightarrow [-1,1]^2
\]

and once we specify $J_0(s,t,m)$ (for $m\in M$) a family of almost complex structures on $M$, the lift of this family of almost complex structures is given by
\[
J=\begin{pmatrix}
j &0\\
* & J_0(s,t,m)
\end{pmatrix}
\]
Here $*$ is defined by the following: let $\partial_t^\#, \partial_s^\#$ denote the horizontal lifts of $\partial_t,\partial_s$ respectively: then we have $J\partial_s^\# = \partial_t^\#$.

From this, it follows that the required almost complex structures of confinement type exist in great abundance and for the differential we can choose a regular $J$ of confinement type.

We now define the vertical energy of a section of $\mathbb{R}\times \mathcal{E}^k\rightarrow \mathbb{R}\times S^1$.

\begin{definition}
Fix any compatible almost complex structure $J$ on $\mathbb{R}\times \mathcal{E}^k$.
Let $u:\mathbb{R}\times S^1 \rightarrow \mathbb{R}\times\mathcal{E}^k$ be a section of $\mathbb{R}\times \mathcal{E}^k\rightarrow \mathbb{R}\times S^1$. For any vector $v$ in $T(\mathbb{R}\times S^1)$, we denote by $v^\#$ its horizontal lift\footnote{i.e. the only vector in the horizontal distribution that projects to $v$}. The \textbf{vertical energy} of $u$ is then defined as
\[
E(u) = \frac{1}{2}\int_{\mathbb{R}\times S^1}|\partial_s u-\partial_s^\#|^2_{g_J}+|\partial_t u-\partial_t^\#|^2_{g_J} ds\wedge dt,
\]
where $g_J$ is the Riemannian metric on the vertical distribution induced by $J$ and $d\alpha^k_\delta$.
\end{definition}
\begin{remark}
    With the above notation, the curvature form $R(u, v)$ is defined as 
    \[
    R(u, v) = -d\alpha_\delta^k(u^\#, v^\#).
    \]
\end{remark}

\begin{proposition}
Fix any compatible almost complex structure $J$ on $\RR \times \mathcal{E}^k$. Suppose the section $u:\mathbb{R}\times S^1 \rightarrow \mathbb{R} \times \mathcal{E}^k$ is $J$-holomorphic: then its vertical energy is given by
\[
E(u) = \int_{\mathbb{R}\times S^1}u^*d\alpha_\delta^k + \int_{\mathbb{R}\times S^1} R(u)(\partial_s,\partial_t)ds \; dt
\]
\end{proposition}
\begin{proof}
See Lemma 8.2.9 of \cite{McDuff_Salamon}.
\end{proof}

\begin{proposition}
\label{prop: vanishing curvature for E^k}
For the vertical 2-form given by $d\alpha^k_\delta$ on $\mathbb{R}\times \mathcal{E}^k$, the curvature 2-form $R$ vanishes identically.
\end{proposition}
\begin{proof}
By our definition, the $2$-form $d\alpha_\delta^k$ is $s$-independent. On the other hand, the horizontal lift $\partial_s^\#$ is simply $\partial_s$. It follows that $\partial_s$ lies in the kernel of $R$, and hence $R(u)(\partial_s, \partial_t)$ vanishes for any $J$-holomorphic $u$. 
\end{proof}

Since the curvature vanishes, we can take the following definition of the global action functional. Suppose $\gamma$ is a time-1 periodic orbit of $\alpha^k_\delta$: then we define its action to be
\[
\mathcal{A}(\gamma):= \int_{\gamma} \alpha^k_\delta.
\]

This has the property that if $u:\mathbb{R}\times S^1\rightarrow \mathbb{R}\times \mathcal{E}^k$ is a $J$-holomorphic section positively asymptotic (i.e. $s\rightarrow+\infty$) to the 1-periodic orbit $\gamma_+$ and negatively asymptotic to the 1-periodic orbit $\gamma_-$, then
\[
E(u) = \mathcal{A}(\gamma_+) - \mathcal{A}(\gamma_-),
\]
with the property that $E(u)\geq 0$ and equality holds if and only if $u$ is a horizontal section, i.e., a trivial cylinder over a fixed periodic orbit.

With the above data specified, we first choose a regular confinement-type almost complex structure $J$ such that all Fredholm index-1 non-horizontal $J$-holomorphic sections are transversely cut out. Then we  define the differential to be
\begin{equation}\label{count}
    \langle d x,y \rangle := \#_{\ZZ_2} (\mathcal{M}_{y,x}/\mathbb{R}).
\end{equation}
We shall see below that choosing $J$ to be of confinement-type means that $J$-holomorphic curves do not escape to the boundary.
The apriori energy bound and the lack of bubbling (because $M$ is exact) means the above count is finite (and well-defined even over $\mathbb{Z}$ coefficients). The usual gluing argument shows that $d^2=0$ and the resulting fixed point Floer cohomology is independent of choice of regular confinement-type $J$.

\section{Confinement for the differential} \label{sec:confinement_differential}

We fix an almost complex structure $J$ of confinement type on $\mathbb{R}\times \mathcal{E}^k$ and prove the confinement results necessary for our splitting formula.

The main theorem of the section is the following:
\begin{theorem} \label{thm:confinement_differential}
There exists some sufficiently large $C$ so that for all sufficiently small $\delta, \epsilon, \epsilon_G >0$, we can choose Floer data (whose dependence on $\delta, \epsilon, \epsilon_{G}$ is specified in the proofs below) so that
\begin{enumerate}
\item If a $J$-holomorphic section has both punctures asymptotic to orbits in the Morse region, then it is entirely contained in the Morse region;
\item If a $J$-holomorphic section has both punctures asymptotic to orbits in the Dehn twist region, it is entirely contained in the Dehn twist region;
\item If a $J$-holomorphic section has a puncture asymptotic to an orbit in the strict Morse region, then it is entirely contained in the Morse region. If a $J$-holomorphic section has a puncture asymptotic to an orbit in the strict Dehn twist region, then it is entirely contained in the Dehn twist region.

\end{enumerate}
\end{theorem}

We begin with the fact that $J$-holomorphic curves cannot escape through the boundary. 
Then we use a version of the Hein inequality \cite{Hein} to confine $J$-holomorphic curves with both ends in the Morse region. Finally, we prove a local energy inequality for $J$-holomorphic curves that intersect both the Morse region and the Dehn twist region, using which we conclude the theorem. 

\begin{proposition}
Suppose $J$ is a confinement-type almost complex structure and $u:\mathbb{R}\times S^1 \to \mathbb{R}\times \mathcal{E}^k$ is a $J$-holomorphic section. Then there is a collar neighborhood of the boundary of $M$, which we write as $(\epsilon_b/2,\epsilon_b]\times \partial M$ with coordinates $(\mu, y)$, so that if we consider the inclusion
\[
\mathbb{R}\times S^1 \times (\epsilon_b/2,\epsilon_b]\times \partial M \subset \mathbb{R}\times \mathcal{E}^k,
\]
then we have
\[
u\cap (\mathbb{R}\times S^1 \times (\epsilon_b/2,\epsilon_b]\times \partial M)  = \emptyset.
\]
\end{proposition}

The proposition follows immediately from the maximum principle:

\begin{lemma}\label{mp}
(Maximum principle) Let $J$ be a confinement-type almost complex structure.  Let $V$ denote an open subset of $\mathbb{R}\times S^1$ with coordinates $(s,t)$. Let $\tilde{u}:V\to \mathcal{E}^k$ denote a $J$-holomorphic section, which has the form $(s,t)\mapsto (s,t,\mu_i(s,t),y(s,t)) \subset \mathbb{R}\times S^1 \times (-\epsilon,0]\times \partial M$. Then $\mu_i(s,t)$ is a harmonic function.
\end{lemma}
\begin{proof}
Recall that near the $i$-th boundary component, the Hamiltonian function $kf_1$ has the form $k\theta_i\mu_i$. The horizontal lifts are
\begin{align}
    \partial_s^\#=\partial_s,\ & {} & \partial_t^\#=\partial_t-k\theta_i R,
\end{align}
where $R$ is the Reeb vector field of the contact manifold $\partial M$. The assumption that $\tilde{u}$ is a $J$-holomorphic section implies that
\begin{equation}
\begin{cases}
\frac{\partial \mu_i}{\partial t}+\lambda_P(\frac{\partial y}{\partial s})=0,\\
\frac{\partial \mu_i}{\partial s}-\lambda_P(\frac{\partial y}{\partial t})+k\theta_i=0.
\end{cases}
\end{equation}
The conclusion then follows by a simple calculation.
\end{proof}
\subsection{Hein inequality}

Recall that near the intersection of the Morse region and the Dehn twist region, the bundle $\mathbb{R}\times \mathcal{E}^k$ takes the form
\[
\mathbb{R} \times S^1 \times (1-\epsilon, 1+\epsilon) \times S^*S^n
\]
with the stable Hamiltonian structure $(\alpha, d(\theta_{T^{*}S^n} - kG_1(\mu)d\alpha))$ for $\mu \notin (1-\delta,1+\delta)$.

In this subsection we will fix some specific quantitative data on $G_1$. Recall that for now this function depends on $\epsilon, \epsilon_G$. 

We fix a large enough $C$ so that if we consider the region $\mathbb{R} \times S^1 \times [1-1/C, 1-1/C^2] \times S^*S^n$, there are no periodic orbits of $kG_1$ in this region, and $kG_1$ is independent of $\epsilon, \epsilon_G$ in this region. This choice of $C$ is large but independent of how small we will choose $\epsilon$. For $C$ large, we can and will think of $\mathbb{R}\times S^1\times (1-1/C,\epsilon)_\mu \times S^*S^n$ as an extension of the Morse region, which by definition only include the region $\mu\geq 1-\epsilon$. We arrange so that the only periodic orbits in the region $\mu\in [1-1/C,1+\epsilon]$ comes from the Morse perturbations of the orbits in the $\mu=1$ locus.

We also arrange so that for all points in the Morse region to the right of $\mu=1+ \epsilon$, the function of $f_1$ takes values in $[\frac{1}{2}\epsilon_G \epsilon ^2, \frac{3}{2}\epsilon_G \epsilon ^2]$.

For ease of terminology, we will refer to critical points in the Morse region but not over $\mu=1$ as being \textit{strictly} in the Morse region. Likewise, we will refer to critical points in the Dehn twist region but not coming from those in the $\mu=1$ locus as being in the \textit{strict Dehn twist region}.

\begin{proposition} \label{prop_hein_differential}
Fix $C$ as above, and suppose $x$ and $y$ are both critical points in the Morse region (not necessarily strictly in the Morse region). Fix an compatible almost complex structure $J$, then for $\epsilon>0$ sufficiently small, any $u\in \mathcal{M}^J(x,y)$ does not intersect the hypersurface
\[
\mathbb{R} \times S^1 \times \{1-1/C\} \times S^*S^n
\]
\end{proposition}
This in particular means that the $J$-holomorphic curves that begin and end in the Morse region more or less stay completely within the Morse region. If we take $\epsilon_G$ and $\epsilon$ sufficiently small, it follows from standard arguments that the $J$-holomorphic curves that begin and end in the Morse region are identified with gradient flow lines.

The key input to the above proposition is an energy estimate of Hein, which we recall for the convenience of the reader.

Let $(M,\omega)$ be a symplectic manifold and $J$ a (potentially time dependent) compatible almost complex structure. Let $H$ denote a (nondegenerate) Hamiltonian on $M$. Hein considers Floer cylinders $u:\mathbb{R}\times S^1 \rightarrow M$ satisfying Floer's equation
\[
\partial_s u +J(\partial_t u -X_H)=0
\]
asymptotic to time-1 periodic orbits of the Hamiltonian. She defines the energy of the map to be
\[
E_H(u):=\int_{\mathbb{R}\times S^1} \|\partial_s u\|^2 ds \wedge dt
\]
where $\|-\|$ is the metric determined by the symplectic form and the almost complex structure. The main proposition of Hein is as follows
\begin{proposition}[Proposition 3.2 \cite{Hein}] \label{prop:Hein}
Let $K$ be a non-degenerate Hamiltonian and let $U$ and $V$
be open sets that are both bounded by level sets of $K$. Assume furthermore
that $K$ does not have $1$-periodic orbits in
$\overline{V} \setminus  U$ and is autonomous on this
shell. Let $u: S^1 \times \mathbb{R} \to M$ be a Floer trajectory that intersects $\partial U$ and $\partial V$.
Then there is a constant $\epsilon_H > 0$, only depending on the open sets $U$ and $V$, as well as
the restriction of the Hamiltonian $K$ and the almost complex structure $J_t$ to $V\setminus U$, such that $E_H(u)|_{U\setminus V} > \epsilon_H$.
\end{proposition}

\begin{remark}
Even though Hein considers entire Floer trajectories $u:\mathbb{R}\times S^1 \to M$, her proof only examines the parts of $u$ that intersect $V\setminus U$.
\end{remark}

\begin{proof}[Proof of Proposition \ref{prop_hein_differential}]
Suppose $u$ is a map from $x$ to $y$, then its vertical energy $E(u)$ is pointwise positive and equal to $\mathcal{A}(x)-\mathcal{A}(y)\leq 10k\epsilon_G \epsilon^2$. Consider the collar neighborhood of the mapping torus near the intersection of the Dehn twist region and the Morse region, which looks like 
\[
\mathbb{R} \times S^1 \times (1-1/C, 1+\epsilon) \times S^*S^n
\]

The intersection of the image of $u$ with this region can be identified with a map $u':S \rightarrow (1-1/C, 1+\epsilon) \times S^*S^n$ satisfying the Floer equation. We suppose for the sake of contradiction that $u$ intersects the hypersurface $\mathbb{R}\times S^1 \times \{1-1/C\} \times S^*S^n$ (i.e., $u$ crosses into the Dehn twist region). If we further restrict to 
\[
(1-1/C, 1-1/C^2) \times S^*S^n
\]
we can take this Hamiltonian to be $-kG_1(\mu)$. Restricted to $S$, Hein's notion of energy for the associated Floer trajectory can be identified pointwise with our notion of energy $E(u)$, and integrated over $S$ it is bounded above by $10k\epsilon_G\epsilon^2$. We apply Hein's proposition by taking $V\setminus U$ to be  $(1-1/C, 1-1/C^2) \times S^*S^n$ and conclude that no such curve can exist if we take $\epsilon>0$ small enough.
\end{proof}
Now we have proved $J$-holmorphic curves that begin and end in the Morse region cannot cross the $\mu=1-1/C$ barrier, we can use the maximum principle (an adaptation of lemma \ref{mp} to allow $kf_1$ to be any function of $\mu$ instead of just the linear case) to further sharpen this result. Note we needed the $\mu=1-1/C$ bound first before we can apply the maximum principle.
\begin{corollary}
For any $0<\eta< 1-\delta$, suppose $u\in \mathcal{M}^J(x,y)$ and $x,y$ are both in the Morse region, then
\[
u\cap \mathbb{R}\times \{\eta\}\times S^*S^n =\emptyset.
\]
\end{corollary}

\subsection{Local energy inequality}

We next prove a local energy inequality which will be key to understanding $J$-holomorphic curves that pass between the Morse and the Dehn twist regions. For this section we fix $C,\epsilon, \epsilon_G$ and fix a compatible $J$ of confinement type as in the previous subsection so that Proposition \ref{prop_hein_differential} holds. The only property we will need about $J$ in this subsection will be that it is confinement type.

We pick $\mu_0\in [1-\epsilon,1+\epsilon]\setminus (1-\delta,1+\delta)$. Suppose $u:\mathbb{R} \times S^1 \rightarrow \mathbb{R} \times \mathcal{E}^k$ intersects the $\mu=\mu_0$ hypersurface transversely.

\begin{lemma} \label{lem:local_energy_inequality_differential}
    Suppose we orient $\Gamma:=u\cap \{\mu=\mu_0\}$ with the boundary orientation of $u\cap \{\mu\leq \mu_0\}$. 
    In the neighborhood we have fixed above, we can write $\theta_{T^*S^n}=\mu \lambda_{S^*S^n}$, where $\lambda_{S^*S^n}$ is the contact form given by restricting $\theta_{T^*S^n}$ to the cosphere bundle. 
    Then we have the following local energy inequality
    \[
    \int_\Gamma \lambda_{S^*S^n}+k\epsilon_G(\mu_0-1) d\alpha \ge 0.
    \]
    where equality holds if and only if $\Gamma$ is empty.
\end{lemma}
\begin{proof}

The proof is a generalization of the one in \cite{hutchings2005periodic}. We parameterize $\Gamma$ as $(s(\tau), \alpha (\tau), \mu_0, y(\tau))\in \mathbb{R} \times S^1 \times \{\mu_0\}\times S^*S^n$. Write $\Gamma'(\tau) = aR+v+b\partial_s+c\partial_\alpha$ where $R$ is the Reeb vector of $(S^*S^n,\lambda_{S^*S^n})$, and $v\in TY$ is in the contact distribution of $(S^*S^n,\lambda_{S^*S^n})$. By our assumptions on $J$, we have 
    \[
    J(\partial_s)=\partial_\alpha-k(\mu_0-1)R,
    \]
    and
    \[
    J(\partial_\alpha) = -\partial_s-k\epsilon_G(\mu_0-1)\partial_\mu.
    \]
    Consequently, we have $$J\Gamma'(\tau) = (-a-k\epsilon_G c(\mu_0-1))\partial_\mu+Jv-k\epsilon_G b(\mu_0-1)R+b\partial_\alpha-c\partial_s.$$ By our orientation convention, it must hold that
    \[
    -a-k\epsilon_G c(\mu_0-1)<0
    \]
    unless $\Gamma$ is empty. Notice that $a$ is just $\lambda_{S^*S^n}(\Gamma'(\tau))$ and $c$ is just $d\alpha (\Gamma'(\tau))$. Integrating the above inequality over $\Gamma$, we get the desired formula.
\end{proof}

Using the above local energy inequality we prove three cases of the confinement principle. 

\begin{lemma} \label{lem:confinement_differential_twist_region}
If $x,y$ are both orbits in the Dehn twist region, then the conclusion of Theorem \ref{thm:confinement_differential} holds.
\end{lemma}
\begin{proof}
Consider $\Gamma:=C\cap \{\mu=1+\epsilon\}$ with $\epsilon$ small and chosen so that this intersection is transverse. From the assumptions on orbits, we have
\[
\int_{\Gamma} d\alpha =0.
\]
The energy inequality then tells us 
\[
\int_{C\cap \{x\geq 1+\epsilon\}} d\alpha^k =  -\int_{\Gamma} \alpha^k <0,
\]
where $\alpha^k = \mu\lambda_{S^*S^n} + \epsilon_G(\frac{k}{2}(\mu-1)^2)d\alpha$, which is a contradiction.
\end{proof}

\begin{remark}
Note this statement is independent of the Morse function in the Morse region; this is a manifestation of a maximum principle in symplectic homology. 
\end{remark}

\begin{lemma}
If $x$ is in the Dehn twist region but not in the Morse region, and $y$ is in the Morse region but not in the Dehn twist region, or vice versa, then the moduli space $\mathcal{M}^J(x,y)$ is empty.
\end{lemma}
\begin{proof}
Let us first deal with the case where $x$ is in the twist region and $y$ is in the Morse region. Take $\Gamma:= C\cap \{\mu=1+\epsilon\}$. The local energy inequality tells us that
\[
\int_{\Gamma} (\lambda_{S^*S^n} + k\epsilon_G\epsilon d\alpha) \geq 0.
\]
The assumption on the two orbits gives
\[
\int_{\Gamma}d\alpha = -1, \quad \int_{y} d\alpha =+1.
\]
So consider
\begin{align*}
0 &\leq \int_{C\cap \{\mu\geq 1+\epsilon\}} d\alpha^k \\
&= -\int_{\Gamma} \mu \lambda_{S^*S^n} - \int_{\Gamma}kf_1d\alpha-\int_{y}kf_1d\alpha \\
& \leq -(1+\epsilon)k\epsilon_G\epsilon +\frac{1}{2}k\epsilon_G\epsilon^2 - kf_1(y)
\end{align*}
which implies
\[
f_1(y) \leq -(1+\epsilon)\epsilon_G\epsilon +\frac{1}{2} \epsilon^2\epsilon_G
\]
which we can use to arrive at a contradiction by choosing $f_1$ appropriately. In particular, we have chosen $f_1$ to be positive and to have $C^0$ norm around $\epsilon_G\epsilon^2$ in all of the Morse regions.

We next consider what happens when $x$ is in the Morse region and $y$ is in the Dehn twist region.

In this case we let $\Gamma:= C \cap \{ \mu=1-\epsilon\}$. 
The local energy inequality implies
\[
\int_{\Gamma}\lambda_{S^*S^n}-k\epsilon_G\epsilon d\alpha \geq 0.
\]

Again, the assumption on the two orbits gives
\[
\int_{\Gamma} d\alpha =1, \quad \int_{x} d\alpha = +1.
\]
Then consider
\begin{align*}
0 &\leq \int_{C\cap \{\mu \geq 1-\epsilon\}} d\alpha^k \\
&= -\int_{\Gamma} \mu \lambda_{S^*S^n} - \int_\Gamma \epsilon_G k f_1 d\alpha + \int_x \epsilon_G k f_1 d\alpha\\
& \leq -k\epsilon_G\epsilon(1-\epsilon) - \frac{1}{2} k\epsilon_G \epsilon ^2 + k\epsilon_G f_1(x)
\end{align*}
which implies 
\[
f_1(x) \geq \epsilon_G (\epsilon(1-\epsilon) + \frac{1}{2} \epsilon^2),
\]
a contradiction.
\end{proof}

With the above we conclude the proof of Theorem \ref{thm:confinement_differential}.

\subsection{Wrapped Dehn twists}
In this section, we give the definition of fixed point Floer cohomology of wrapped Dehn twists on $T^*S^n$, largely following the conventions of \cite{ganatra_thesis}.

Suppose $(D^*S^n,\lambda)$ is the standard Liouville structure, with boundary $S^*S^n$. Denote its completion into a Liouville manifold by $(T^*S^n,\lambda)$, which we think of as attaching the collar $[1,\infty)_\mu\times S^*S^n$.

To construct the mapping tori for wrapped Dehn twists, we consider the adapted model iterated Dehn twist $\phi_{kG_1}$ on $D^*_{\mu \leq 1+\epsilon}S^n$ and take the mapping torus with the standard stable Hamiltonian structure. The Morse region is given by $T^*_{\mu\geq 1-\epsilon}S^n$ and the mapping torus is identified with $S^1\times T^*_{\mu\geq 1-\epsilon}S^n$. Instead of taking a $C^2$-small function $f_1$, we take the stable Hamiltonian structure $(d\alpha, d\theta_{T^*S^n} + H(\mu)d\alpha)$ where $H(\mu) = \frac{1}{2}\epsilon_G(\mu-1)^2$ near the Dehn twist region and is equal to $\frac{1}{2}\mu^2$ near infinity. We then use small time-dependent perturbations to break the Morse-Bott degeneracy. We denote the resulting symplectomorphism by
$\phi^k_\infty$ and call it the model \textit{wrapped Dehn twist}. We write $Y_{\phi^k_\infty}$ for the mapping torus of the associated wrapped Dehn twist.

We recall from \cite{ganatra_thesis} that an almost complex structure on $T^*S^n$ is called \textit{rescaled contact type} if, outside some compact set,
\[
\frac{1}{\mu} \theta_{T^*S^n} \circ J =d\mu
\]
and called \textit{$c$-rescaled contact type} if for some positive $c$ 
\[
\frac{c}{\mu} \theta_{T^*S^n} \circ J =d\mu.
\]
outside of some compact set.

We choose a (regular) almost complex structure on $Y_{\phi^k_\infty}$ of confinement type near the boundary of the Dehn twist region, which is the lift of a rescaled contact type almost complex structure on $T^*S^n$. With these data fixed, we may define fixed point Floer cohomology of wrapped Dehn twists. We note that if two orbits are in the Dehn twist region, the differential connecting them is also entirely contained in the Dehn twist region, by Lemma \ref{lem:confinement_differential_twist_region}. This in particular means we can view curves in the Dehn twist region as a finite sector of fixed point Floer cohomology of wrapped Dehn twists.\footnote{To show this differential is well defined we do need a maximum-principle-type argument to show that if we fix an input $x$ and output $y$, the curves connecting them do not escape to infinity. We can do this either using the dissipative estimates of \cite{ganatra_thesis}, or in this case leveraging the fact the curvature on $Y_{\phi^k_\infty}$ is zero, then noticing that Lemma \ref{lem:local_energy_inequality_differential} holds for rescaled contact type almost structures and applying Lemma \ref{lem:confinement_differential_twist_region}. }

For even iterations of wrapped Dehn twists, $HF^*(\phi_{\infty}^{2n})$ can be identified with $SH^*_{FL}(T^*S^n)$, up to grading shifts. We shall make further use of this correspondence in future work.

\section{The product structure for Dehn twists} \label{sec:prod_str_Dehn_twists}

We now define the count of $J$-holomorphic curves required to define the product structure on fixed point Floer cohomology.

Recall the locally Hamiltonian fibration $E^{m,n}$ defined in Section \ref{sec:cobordism_of_Dehn_twists}. Near the positive puncture of $B_3$, this looks like the symplectization of the stable Hamiltonian structure $(dt_\infty, d\alpha^{m.n})$, and near the two negative punctures, they are given by the symplectizations of the stable Hamiltonian structures $(dt_1,d\alpha^m)$ and $(dt_2, d\alpha^n)$. The main issue with these stable Hamiltonian structures is that they are Morse-Bott instead of Morse, so to properly define the product operation, we now perturb them to be Morse.

Near the positive puncture, the symplectic fibration looks like a symplectization, which we can write $[C,\infty)\times \mathcal{E}^{m+n}$ with $(dt_\infty, d\alpha^{m+n})$. In the interval $[C',C'+1]$ (for large $C'$) make a size-$\delta$ (in the $C^2$ norm, say) interpolation from $\alpha^{m+n}$ to $\alpha^{m+n}_\delta$. Do the same for the other two punctures of $B_3$. We call the resulting bundle $E^{m,n}_\delta\rightarrow B_3$, and the global 1-form we will denote by $\lambda^{m,n}_\delta$.

To consider curve counts in this symplectic fibration, we equip it with an almost complex structure that is compatible with the symplectic fibration and is of SFT type near the punctures. To be more precise:

\begin{definition}
An almost complex structure $J$ on $E^{m,n}_\delta$ is compatible if the following conditions are satisfied:
\begin{enumerate}
\item Near the punctures of $B_3$ where the symplectic fibration looks like the symplectization of a mapping torus, the almost complex structure is compatible in the sense of Definition \ref{defn:boundary_requirement_J}.

\item Away from the punctures of $B_3$, $J$ is compatible in the sense of a locally Hamiltonian fibration. To be specific, it preserves the vertical and horizontal sub-bundles, restricts to a compatible almost complex structure with respect to $d\lambda_\delta^{m,n}$, and the projection to the base is holomorphic.

\item Near the boundary of the fibers, $J$ is compatible with the symplectization structure.
\end{enumerate}
\end{definition}

For condition (3): in this paper when we say \textit{compatible} we also require compatibility near the boundary of $M$. A neighborhood of the boundary $\partial M$ is identified with the symplectization $(\epsilon_b/2, \epsilon_b]_\mu\times P$ of a contact manifold $P = \partial M$ with the contact form $\lambda_P$ and contact structure $\xi_P$. Condition (3) means that we choose an almost complex structure $J_0$ on $M$ such that near the boundary, $J_0$ is compatible with the described symplectization. In particular, $J_0$ preserves $\xi_P$ and sends $\partial_\mu$ to the Reeb vector field $R_P$ of $(P, \lambda_P)$. The almost complex structure $J$ is chosen to be the compatible almost complex structure that restricts to $J_0$ on the fiber.

For the purpose of enumerating $J$-holomorphic sections, the notion of compatibility is not quite good enough because there may be horizontal sections that are not transverse, which we may not be able to perturb away in the class of compatible almost complex structures (this is not an issue in symplectizations because all horizontal sections there are trivial cylinders). We will adapt the notion of tame almost complex structures as follows:

\begin{definition} \label{defn:tame_for_cobordism}
An almost complex structure $J$ on $E^{m,n}_\delta$ is tame if

\begin{enumerate}
\item Near the punctures of $B_3$, where the symplectic fibration looks like the symplectization of a mapping torus, the almost complex structure is compatible in the sense of Definition \ref{defn:boundary_requirement_J}.
\item Away from the punctures of $B_3$, we only require that $J$ be tamed by the ambient symplectic form $\Omega_{m,n} =\pi^{m,n *}\omega_{B_3} + d\lambda^{m,n}_\delta$.
\item The condition (4) on the boundary $P=\partial M$ in Definition \ref{defn:boundary_requirement_J} is satisfied.
\end{enumerate}

\end{definition}

We will prove confinement of $J$-holomorphic curves for compatible almost complex structures, then deduce confinement for nearby tame almost complex structures. With the notion of tame almost complex structures we may define the count that appears in the product:

Fix a tame almost complex structure  on $E^{m,n}_\delta$. Consider $x,y_1,y_2$ time-$1$ periodic orbits of $\alpha^{m+n}_\delta$, $\alpha^m_\delta$ and $\alpha^n_\delta$ respectively. Define

\begin{equation}
    \mathcal{M}^J(x;y_1,y_2): =\left\lbrace u: B_3\to X \;\middle|\;
  \begin{tabular}{@{}l@{}}
    $\pi_X\circ u$ is degree 1,\ $u$ is $J$-holomorphic, and\\
    $u$ is asymptotic to $y_1$, $y_2$ at its negative punctures and \\ $x$ at its positive puncture.
   \end{tabular}
  \right\rbrace
\end{equation}
We define the chain-level dot product as
\[
\langle y_1\bullet y_2,x\rangle = \#_{\ZZ_2} \mathcal{M}(x;y_1,y_2)
\]
A certain amount of work is required to show that for any tame $J$, this count is finite. We shall use a maximum principle to show $J$-holomorphic sections cannot escape from the boundary in Proposition \ref{prop:max_principle_B3}. Taking that for granted for now, to prove compactness, we need to note the Hofer energy of any curve above is bounded uniformly.

Recall the Hofer energy \cite{wendl2016lectures} of a $J$-holomorphic section $u$ is defined as 
    \[
    E^{\mathrm{Hofer}}(u) = \sup_{f\in \mathcal{T}}\int_{B_3}u^*\omega_f
    \]
    Where $$\mathcal{T} = \{f\in C^\infty(\mathbb{R}, (-\epsilon, \epsilon)) \; | \; f'>0 \text{ and } f(x)=x \text{ near } [-\delta, \delta]\}$$ for sufficiently small $\epsilon$ and $\delta$, and
    \[
    \omega_f=\begin{cases}
        \Omega_{m,n} & \text{away from the punctures} \\
        d(f(s_i)dt_i)+d\lambda^{m,n}_\delta & \text{near the three punctures} 
    \end{cases}
    \]
We see by exactness of $\Omega_{m,n}$ that the Hofer energy is uniformly bounded. This means we can either see elements of $\mathcal{M}(x;y_1,y_2)$ break into several SFT buildings near the punctures \cite{SFT_compactness} or bubble off $J$-holomorphic spheres \cite{McDuff_Salamon} in the interior of $B_3$ away from the punctures. First, note that sphere bubbling is not possible since $\omega_f$ is exact away from the punctures, and by a maximum principle all spheres are uniformly bounded away from a neighborhood of the punctures of $B_3$. Then next note that new SFT buildings are not possible for when the index is zero. From this we deduce that when the index is zero, the count of elements of $\mathcal{M}(x;y_1,y_2)$ is finite.

En route to proving confinement for compatible almost complex structures on $E^{m,n}_\delta$, we will actually first prove confinement for $E^{m,n}$ (the Morse-Bott version).

\begin{definition}
We say a compatible almost complex structure $J$ on $E^{m,n}$ is of confinement type if it is the unique lift of a confinement type almost complex structure $J_0$ on $M$,
see Definition \ref{defn:confinement type acs}.
\end{definition}

\begin{definition}
Let $J$ be a compatible almost complex structure on $E^{m,n}$ and let $u:B_3 \rightarrow E^{m,n}$ be a $J$-holomorphic section, then we define the vertical energy of $u$ to be
\begin{equation}
    E(u)=\frac{1}{2}\int_{B_3}|\partial_s u-\partial_s^\#|^2_{g_J}+|\partial_t u-\partial_t^\#|^2_{g_J} ds\wedge dt
\end{equation}
Here $(s,t)$ denotes a patch of holomorphic coordinates on $B_3$ and
$\partial_s^\#$ and $\partial_t^\#$ are the horizontal lifts of the vector fields $\partial_s$ and $\partial_t$, respectively, and $g_J$ denotes the metric on the vertical distribution determined by $d\lambda^{m,n}$ and $J$.
\end{definition}

\subsection{Confinement for the Dehn twist in the Morse-Bott setting} \label{sec:confinement_Morse_Bott}

We now turn to the confinement of $J$-holomorphic curves for the product. We shall first show the no crossing theorem for the Morse-Bott setting, for the bundle $E^{m,n}$. Then we will deduce the nondegenerate case from the correspondence between $J$-holomorphic curves in the nondegenerate setting and cascades in the Morse-Bott setting. 

We will still be applying the overall strategy for confinement by leveraging both the local energy inequality and Hein's energy inequality. A main technical issue is constructing the local Hamiltonian fibration so that the curvature vanishes: this allows us to read off the energy of a $J$-holomorphic curve from the action of the endpoint orbits. Since all of our confinement results require controlling the action of orbits, this part is essential.

We first verify 
\begin{proposition}
The locally Hamiltonian fibration $E^{m,n}$ has vanishing curvature.
\end{proposition}
\begin{proof}
Recall that $E^{m, n}$ is the pullback bundle $g_{m, n}^*\mathcal{E}^1$. Denote by $\tilde{g}_{m, n}$ the map $E^{m, n}\to \mathbb{R}\times\mathcal{E}^1$ that covers $g_{m, n}: B_3\to\mathbb{R}\times S^1$ (notice that previously $g_{m, n}$ was defined as a map from $B_3$ to $S^1$: here we use the same notation to denote the composition of $g_{m, n}$ with the inclusion sending $S^1\to S^1\times\{0\}\subset \mathbb{R}\times S^1$). By the argument of Proposition \ref{prop: vanishing curvature for E^k}, the curvature form for $d\alpha^1$ vanishes on the bundle $\mathbb{R}\times \mathcal{E}^1$. Now for any vectors $u$ and $v$ in $T_{p_0} B_3$, denote by $u_0$, $v_0$ the vectors $dg_{m, n}(u)$ and $dg_{m, n}(v)$, and $u_0^\#$, $v_0^\#$ their horizontal lifts in $T_{\tilde{p}} \mathcal{E}^1$. We then have
\[
d\alpha^1(u_0^\#, v_0^\#) = 0.
\]

Let $u^\#$ and $v^\#$ be the horizontal lifts of $u$ and $v$ in $T_{\tilde{p}_0}E^{m,n}$ with respect to $d\lambda^{m,n} = \tilde{g}_{m,n}^*d\alpha^1$, then we have
\[
d\tilde{g}_{m,n}(u^\#) = u_0^\#, \quad d\tilde{g}_{m,n}(v^\#) = v_0^\#.
\]
To see this, first notice that $u_0^\#$ and $v_0^\#$ as defined above project to $u_0$ and $v_0$. To verify that they both lie in the horizontal distribution of $\mathbb{R}\times \mathcal{E}^1$, take any vertical vector $w_0$ in $T_{\tilde{p}_0}\mathcal{E}^1$ and consider the unique vertical vector $w\in T_pE^{m,n}$ such that $d\tilde{g}_{m,n}(w) = w_0$ (since $d\tilde{g}_{m,n}$ is an isomorphism restricted to the vertical distributions). We now have
\[
d\alpha^1(u_0^\#, w_0) = d\alpha^1(d\tilde{g}_{m,n}(u^\#), d\tilde{g}_{m,n}(w)) = d\lambda^{m,n}(u^\#, w) = 0.
\]

In particular, we have 
\[
d\lambda^{m,n}(u^\#, v^\#) = d\alpha^1(d\tilde{g}_{m,n}(u^\#), d\tilde{g}_{m,n}(v^\#)) = d\alpha^1(u_0^\#, v_0^\#) = 0.
\]

This shows that the curvature $R(u, v) = -d\lambda^{m,n}(u^\#, v^\#)$ vanishes.
\end{proof}

\begin{corollary}
Let $J$ denote a compatible almost complex structure on $E^{m,n}$.
Let $u:B_3\rightarrow E^{m,n}$ denote a $J$-holomorphic section positively asymptotic to $x$, and negatively asymptotic to $y_1$ and $y_2$, then
\[
E(u) = \mathcal{A}(x)-\mathcal{A}(y_1)-\mathcal{A}(y_2) 
\]
\end{corollary}
\begin{proof}
Because the curvature vanishes, we have
\[
E(u) =\int_{B_3}u^*d\lambda^{m,n},
\]
and we apply Stokes' theorem.
\end{proof}

The next proposition shows that all $J$-holomorphic sections that we count do not approach $\partial M$.
\begin{proposition} \label{prop:max_principle_B3}
Let $J$ be a compatible almost complex structure of confinement type, and let $u$ denote a $J$-holomorphic section. Near the boundary of $M$ the symplectic fibration looks like
\[
B_3 \times (\epsilon_b/2,\epsilon_b]\times \partial M.
\]
Then $u$ does not intersect $B_3 \times \{\eta\} \times \partial M$ for any $\eta \in (\epsilon_b/2,\epsilon_b]$
\end{proposition}
\begin{proof}
Let $(s, t)$ be local conformal coordinates on $B_3$, and let $dg_{m, n} = Fds+Gdt$ with respect to these local coordinates. Suppose $u$ has the form $(s, t)\mapsto (s, t, \mu_i, y)$ with respect to these local coordinates, near a component of $\partial M$ where $f_1 = f_1(\mu_i)$. It then follows with these local coordinates that
\[
\partial_s^\# = \partial_s - f_1'(\mu_i)FR, \quad \partial_t^\# = \partial_t-f_1'(\mu_i)GR,
\]
where $R$ is the Reeb vector field of $(\partial M = P, \lambda_P)$. The compatibility assumption implies that $J(\partial_s^\#) = \partial_t^\#$.

On the other hand, the assumption that $u$ is $J$-holomorphic implies that $J(\partial_s + \frac{\partial\mu_i}{\partial s}\partial_{\mu_i}+\lambda_P(\frac{\partial y}{\partial s})R) = \partial_t + \frac{\partial\mu_i}{\partial t}\partial_{\mu_i}+\lambda_P(\frac{\partial y}{\partial t})R$, which then implies that
\begin{equation}
\begin{cases}
\frac{\partial \mu_i}{\partial t} = -\lambda_P(\frac{\partial y}{\partial s})-f_1'(\mu_i) F\\
\frac{\partial \mu_i}{\partial s} =\lambda_P(\frac{\partial y}{\partial t})+f_1'(\mu_i) G.
\end{cases}
\end{equation}
Since $dg_{m,n}$ is closed, we have $F_t = G_s$. It then follows that $\mu_i$ is a harmonic function, and the conclusion follows from a maximum principle.
\end{proof}
\subsubsection{Confinement using Hein's methods}

We now study the confinement of $J$-holomorphic pairs of pants using Hein's methods. To start, we review the relevant structure in the Morse region.

Recall that we can write the symplectic fibration over the Morse region as 
\[
B_3 \times M_{\mu\geq 1-\epsilon}
\]
with the vertical two-form given by
\[
\lambda^{m,n} = d(\theta_M + f_1 dg_{m,n})
\]
where $f_1= -G_1$ for $\mu \in [1-\epsilon, 1+\epsilon]$ and we take $f_1$ to be a Morse function whose value is between $\frac{1}{2}\epsilon_G\epsilon^2$ and $\frac{3}{2}\epsilon_G \epsilon^2$ on the rest of the Morse region.

We fix some $C\gg 1$ so that there are no periodic orbits of $-(m+n)G_1$ in the interval
\[
\mu \in [1-1/C, 1)
\]
Fix an almost complex structure $J_0$ on the fiber $[1-1/C,1-1/C^2]_\mu \times S^*S^n$ that is compatible with the symplectization structure on $[1-1/C,1-1/C^2]_\mu \times S^*S^n$ as in Definition \ref{defn:confinement type acs}, we lift it to almost complex structures of confinement type in bundles $\mathbb{R} \times \mathcal{E}^*$\footnote{Technically $J_0$ only specifies what the almost complex structure is on $\mathbb{R}\times \mathcal{E}^*$ on the subset $\mathbb{R}\times S^1 \times [1-1/C, 1-1/C^2] \times S^*S^n$, but this is enough for our purposes.} for $*=m,n,m+n$ and we denote the lifted almost complex structures by $J_\infty^*$ for $*=m,n,m+n$.

 Then there is a fixed $\epsilon_C$ so that Hein's lemma applies:
consider $J$-holomorphic sections of $\mathbb{R}\times \mathcal{E}^*$ (here $*=m,n, m+n$) and suppose $D$ is a subset of $\RR \times S^1$: then any $J$-holomorphic curve $u:D \rightarrow D \times \mathcal{E}^*$ which intersects both $\mu=1-1/C$ and $\mu=1-1/C^2$ must have vertical energy
$>\epsilon_C$.

Now we are ready to state the main confinement proposition using Hein's inequality.

\begin{proposition} \label{prop:morse_bott_all_morse}
Suppose $x,y_1,y_2$ are all in the Morse region (this includes the critical points on $\mu=1$). For $\epsilon$ and $\epsilon_G$ sufficiently small, we can find confinement-type almost complex structure $J$ so that
any element $u\in \mathcal{M}(x;y_1,y_2)$ does not intersect the hypersurface $\mu = 1-1/C$. The conditions we put on $J$ are only in the region $1-1/C\leq \mu\leq1+\epsilon$.
\end{proposition}
\begin{proof}
\textbf{Step 0:}
The proof is by contradiction.
We take a sequence of positive constants $\epsilon_k, \epsilon_{G,k} \rightarrow 0$ as $k \to \infty$. We assume $-G_{1,k} = \frac{1}{2}\epsilon_{G,k} (\mu-1)^2$ on $[1-1/C^4,1+\epsilon]$, and that $G_{1,k}$ is fixed and does not depend on $k$ for $\mu \leq 1-1/C^2$ and agrees with the $G_1$ we specified at the beginning of this subsection (and between these regions we interpolate via a smooth convex function). We assume the only orbits in the region $\mu\in [1-1/C,1+\epsilon]$ are the ones at the $\mu=1$ locus.

The resulting $G_{1,\infty}$ is then smooth and vanishes in near $\mu=1$. We also choose a sequence of confinement type almost complex structures
as follows. Over the $\mu\in [1-1/C,1+\epsilon]$ of the $E^{m,n}$, the bundle is 
\[
B_3 \times [1-1/C,1+\epsilon] \times S^*S^n
\]
with $\lambda^{m,n}=d(\theta_M-G_{1,k}dg_{m,n})$. We fix $J_0$ as a symplectization compatible almost complex structure on $[1-1/C,1+\epsilon] \times S^*S^n$, and lift it to a compatible almost complex structure $J_k$ on $B_3 \times [1-1/C,1+\epsilon] \times S^*S^n$, and extend this to be a compatible almost complex structure on rest of $E^{m,n}$.

We then obtain a sequence $J_k$ each of confinement type that converges in $C^\infty_{loc}$ as $k \to \infty$ to a fibration-compatible almost complex structure $J_\infty$; the almost complex structures $J_k$ are independent of $k$ in the interval $\mu \in [1-1/C,1-1/C^2]$, and that furthermore, in the fixed cylindrical neighborhoods near the punctures of $B_3$, when further restricted to the region $\mu\in [1-1/C,1-1/C^2]$, the almost complex structure $J_\infty$ agrees with the previously fixed almost complex structures of confinement type $J_\infty^m,J_\infty^n,J_\infty^{m+n}$. \footnote{The properties of $J_k$ and the resulting $J_\infty$ as stated in this paragraph are in fact all that we need in this proof. In the previous paragraph we gave an explicit construction for such $J_k$ and $J_\infty$, but other constructions are certainly possible. In particular we did not really need $J_0$ to be compatible with the symplectization structure on the fiber for all $\mu\in [1-1/C,1+\epsilon]$, we just need it for $\mu\in [1-\epsilon,1+\epsilon]\setminus (1-\delta,1+\delta)$ so that its lift is of confinement type, and the rest of the proof goes through.}

We suppose that for all $k\geq 1$ there exists $u_k\in \mathcal{M}^{J_k}(x;y_1,y_2)$ which intersects $\mu =1-1/C$ (in the $k\rightarrow \infty$ limit there is an identification of the critical points of $f_{1,k}$ for different values of $k$, which we suppress from our notation). 

\textbf{Step 1:}
Our first observation is that the vertical energy of $u_k$ goes to zero as $k \to \infty$, since in the Morse-Bott situation
$E(u_k) = \mathcal{A}(x) -\mathcal{A}(y_1) - \mathcal{A}(y_2)$.

\textbf{Step 2:}
We decompose the domain $B_3$ of the curve $u$ into four pieces, $B_3= C_1\cup C_2 \cup C_3 \cup C_0$, where $C_0$ is a pair of pants with boundary, and $C_i$ are semi-infinite cylinders where $dg_{m,n} = m dt, ndt, (m+n)dt$. In other words, the curves $u_k$ restricted to $C_1, C_2, C_3$ gives $J_k$-holomorphic sections of $\mathbb{R}\times \mathcal{E}^{m},\mathbb{R}\times \mathcal{E}^{n}$ or $\mathbb{R}\times \mathcal{E}^{m+n}$.

\textbf{Step 3:}
We next note the following lemma:
\begin{lemma}
We have $\| du_k\|_{C_0} \leq C'$: in other words, there is a uniform upper bound on the derivative of $u_k$ restricted to $C_0$.
\end{lemma}
To see this, the maps $u_k: C_0\rightarrow X$ 
may be viewed as $J_k$-holomorphic sections of the almost complex fibration $E^{m,n} \rightarrow C_0$. Suppose the lemma does not hold: then we can find $z_k\in C_0$ converging to some $z' \in C_0$ (potentially in the boundary of $C_0$) so that $\|du_k(z_k)\| \rightarrow \infty$. By Gromov compactness \cite{McDuff_Salamon}, this corresponds to a $J_{\infty}$-holomorphic sphere bubble in the fiber over $z\dash$, which we know is impossible since the whole fibration is exact with vertical $2$-form given by $d\lambda ^{m,n}$. Hence this cannot happen, and we have the lemma. 

\textbf{Step 4:} Now by the Arzel\`a-Ascoli theorem, we may pick a subsequence of $\epsilon_k$ (which by an abuse of notation we continue to denote by $\epsilon_k$), so that $u_k: C_0 \rightarrow E^{m,n} $ converges in $C^\infty$ to a $J_{\infty}$-holomorphic map, $u_\infty: C_0 \rightarrow E^{m,n}$ for the almost complex structure $J_\infty$.  Also since $\epsilon_k\rightarrow 0$ as $k \to \infty$, the resulting limit $u_{\infty}$ has zero vertical energy, so it must be a horizontal section.

We next need to understand the map $u_\infty: C_0\rightarrow X$ with zero vertical energy. We can split the fiber into two pieces: the Dehn twist region and the Morse region. 

We first suppose that $u_\infty |_{C_0}$ intersects the Morse region. In the Morse region the bundle splits as $C_0 \times M_{\mu>1-1/C^2}$, and all maps with zero vertical energy are constants, hence $C_0$ must be entirely contained in the Morse region. Then for large enough $k$, the restriction $u_k|_{C_0}$ lives in $\{\mu>1-1/C^2 \}$. This means if $u_k$ intersects $\{\mu=1-1/C\}$, it must come from the restriction of $u_k$ to $C_1,C_2$ or $C_3$. Thus $u_k$ has a cylindrical part that crosses the barrier $\mu \in [1-1/C,1-1/C^2]$. Noting that $J_k$ is independent of $k$ in this region, Hein's lemma (Proposition \ref{prop:Hein}) provides a ($k$-independent) lower bound on the vertical energy of $u_k$, contradicting the fact that it goes to zero.

Next we consider what happens when $u_\infty|_{C_0}$ intersects the Dehn twist region. Recall that the Dehn twist region is denoted by $g_{m,n}^*B^{1}_{\mu\leq 1+\epsilon}$. We can decompose the Dehn twist region into two parts, $g_{m,n}^*B^{1}_{\epsilon''\leq \mu\leq 1+\epsilon}$ and $g_{m,n}^*B^{1}_{ \mu\leq \epsilon''}$. The difference is that $g_{m,n}^*B^{1}_{\epsilon''\leq \mu\leq 1+\epsilon}$ is again a trivial bundle, whereas $g_{m,n}^*B^{1}_{ \mu\leq \epsilon''}$ is not a trivial bundle.
Now we observe that all horizontal sections in $g_{m,n}^*B^{1}_{\epsilon''\leq \mu\leq 1+\epsilon}$ have constant $\mu$ values. This means if $u_\infty|_{C_0}$ intersects $g_{m,n}^*B^{1}_{ \mu\leq \epsilon''}$ then it must be strictly contained in it, and for large enough $k$, this means that $u_k|_{C_1}$ (or $C_2$ or $C_3$) must cross the $\mu =1-1/C$ barrier, a contradiction. 

If $u_\infty|_{C_0}$ is contained in $g_{m,n}^*B^{1}_{\epsilon''\leq \mu\leq 1+\epsilon}$, it must have constant $\mu$ value. If this $\mu$ satisfies $\mu<1-1/C$, the argument is the same as above, and if $\mu >1-1/C$ the argument is the same as if $u_\infty|_{C_0}$ is in the Morse region. The case of $\mu=1-1/C$ can be covered as well: we can always either slightly shrink $C$ or apply the above argument to a slightly larger $C'>C$.
\end{proof}
Another way to state the above proposition is that curves with all three ends in the Morse region stay in the Morse region. In fact the following corollary follows from the maximum principle stated in Proposition \ref{prop:max_principle_B3}:

\begin{corollary}
Let $u\in \mathcal{M}(x;y_1,y_2)$ with $x,y_1,y_2$ all in Morse region. Choosing the same Floer data as above, for any $\mu_0<1$, $u$ does not intersect the $\mu=\mu_0$ hypersurface.
\end{corollary}

\begin{remark}
In some sense the above proves Hein's confinement principle for the pair of pants product on Hamiltonian Floer homology. We used a degeneration argument, but we suspect that a direct proof using energy estimates (and carefully-constructed Hamiltonians as above so the curvature term in the symplectic fibration is small) is also possible.
\end{remark}
\subsubsection{Confinement using local energy inequality}

We assume we have fixed all the data from the previous subsection, i.e. $C,\epsilon,\epsilon_G$ and some almost complex structure $J$ whose behavior in the region $\mu\in [1-1/C,1+\epsilon]$ we have already fixed. In particular this $J$ is of confinement type, which is the only property of $J$ will use in this section in conjunction with constraints on $f_1$.

In analogy with the case for the differential, we have a local energy inequality for the product case.

We pick $\mu_0\in [1-\epsilon,1+\epsilon]$ irrational. Suppose $u:B_3 \rightarrow E^{m,n}$ is a $J$-holomorphic section in the moduli space that intersects the $\mu=\mu_0$ hypersurface transversely.

\begin{lemma}
    Let $u$, $J$ be chosen as above, and orient $\Gamma:=u\cap \{\mu=\mu_0\}$ with the boundary orientation of $u\cap \{\mu\leq \mu_0\}$. 
    In the neighborhood we have fixed above, we can write $\theta_{T^*S^n}=\mu \lambda_{S^*S^n}$, where $\lambda_{S^*S^n}$ is the contact form given by restriction of $\theta_{T^*S^n}$ to the cosphere bundle.
    Then we have the following local energy inequality
    \[
    \int_\Gamma \lambda_{S^*S^n}+\epsilon_G(\mu_0-1) dg_{m,n} \ge 0.
    \]
    where equality holds if and only if $\Gamma$ is empty.
\end{lemma}
\begin{proof}
The proof is similar to the proof of Lemma \ref{lem:local_energy_inequality_differential}. Choose local conformal coordinates $(s, t)$ on $B_3$, and parametrize $\Gamma$ as $(s(\tau), t(\tau), \mu_0, y(\tau))\in B_3 \times \{\mu_0\}\times S^*S^n$. Write $\Gamma'(\tau) = aR+v+b\partial_s+c\partial_t$ where $R$ is the Reeb vector of $(S^*S^n,\lambda_{S^*S^n})$, and $v\in TY$ is in the contact distribution of $(S^*S^n,\lambda_{S^*S^n})$. Assume that $dg_{m,n} = Fds +Gdt$ in the above local coordinates. By our assumptions on $J$, we have the following computations for the horizontal lifts $\partial_s^\#$ and $\partial_t^\#$:
\[
\partial_s^\# = \partial_s - \epsilon_G(\mu_0-1)FR, \quad \; \; \partial_t^\# = \partial_t - \epsilon_G(\mu_0-1)GR.
\]
So we have \[J(\partial_s - \epsilon_G(\mu_0-1)FR) = \partial_t - \epsilon_G(\mu_0-1)GR,\]
and hence
\[
J(\partial_s) = \partial_t-\epsilon_G(\mu_0-1)(GR+F\partial_\mu), \quad J(\partial_t) = -\partial_s + \epsilon_G(\mu_0-1)(FR-G\partial_\mu).
\]

On the other hand, we have
\[
J(\Gamma'(\tau)) = J(aR+v+b\partial_s+c\partial_t) = -a\partial_\mu + J(v) + bJ(\partial_s)+cJ(\partial_t).
\]

Our assumption on the orientation implies that the $\partial_\mu$-coordinate of $J(\Gamma'(\tau))$ is negative, unless $\Gamma$ is empty. Hence we have:
\[
-a-\epsilon_G(\mu_0-1)(bF+cG) < 0.
\]
Finally, notice that $a = \lambda_{S^*S^n}(\Gamma'(\tau))$, and $bF+cG = dg_{m,n}(\Gamma'(\tau))$. The conclusion then follows by integrating the above inequality over $\Gamma$.
\end{proof}

From this, we can deduce the following confinement results. To start, we recall some terminology. The Morse region and the Dehn twist region of $E^{m,n}$ intersect in a neighborhood of the $\mu=1$ region which is diffeomorphic to
\[
B_3 \times (1-\epsilon,1+\epsilon)_\mu \times S^*S^n.
\]

Recall our convention that the periodic orbits in the Morse region include those in this intersection region, and likewise for periodic orbits in the Dehn twist region. The periodic orbits in the Morse region that do not intersect the above intersection region are said to be in the strict Morse region. Likewise, periodic orbits in the Dehn twist region that are not in the above region are said to be in the strict Dehn twist region.

We have also fixed the Morse function in the Morse region. The Morse region looks like
\[
B_3 \times M|_{\mu\geq 1-\epsilon}
\]
with $d\lambda^{m,n} = d\theta_M +df_1dg_{m,n}$. 
We required $f_1 = \frac{1}{2}\epsilon_G(\mu-1)^2$ on $[1-\epsilon,1+\epsilon]$, and on the rest of $M$ we required it to be a $C^2$ small Morse function whose value is between $\frac{1}{2}\epsilon_G \epsilon^2$ and $\frac{3}{2}\epsilon_G \epsilon^2$.

\begin{lemma} \label{lem:confine_product_all_twist}
Suppose $x,y_1,y_2$ are all in the Dehn twist region: then any $C\in \mathcal{M}(x;y_1,y_2)$ is also contained entirely in the Dehn twist region.
\end{lemma}

\begin{proof}
Assume not: then we consider the nonempty intersection $\Gamma:= C\cap\{\mu=1+\epsilon\}$, so that 
\[
\int_\Gamma dg_{m,n}=0
\]
Then we have
\[
\int_{C\cap \{\mu\geq 1+\epsilon\}} d\lambda^{m,n} = -(1+\epsilon)\int_\Gamma \lambda_{S^*S^n} <0,
\]
by the local energy inequality, which is a contradiction.
\end{proof}

\begin{lemma}
If $x$ is in the Dehn twist region, and if at least one of $y_1$ or $y_2$ is in the strict Morse region, then $\mathcal{M}(x;y_1,y_2)$ is empty.
\end{lemma}

\begin{proof}
We separate the proof into two cases.

Suppose $x$ is in the Dehn twist region, and both $y_1$ and $y_2$ are in the strict Morse region. If such a $J$-holomorphic section $C$ exists, we consider the transverse intersection $\Gamma = C\cap \{x=1+\epsilon\}$ for some small positive $\epsilon$. Clearly we have $\int_\Gamma dg_{m,n} = -(m+n)$. The local energy inequality then implies that
\[
\int_\Gamma \lambda_{S^*S^n}\geq \epsilon_G\epsilon (m+n).
\]
On the other hand, we must have $0\leq E(C\cap\{\mu\geq 1+\epsilon\})$, which is to say
\[
0\leq -\int_\Gamma \mu \lambda_{S^*S^n} - \int_\Gamma f_1dg_{m,n} - m f_1(y_1) - n f_1(y_2).
\]
So we have $m f_1(y_1) + n f_1(y_2)\leq (m+n)(-\epsilon-\frac{1}{2}\epsilon^2)\epsilon_G$, which is a contradiction if we choose $f_1$ to be a sufficiently small perturbation of $\frac{1}{2}\epsilon^2\epsilon_G$ in the Morse region.

Next suppose $x$ is in the Dehn twist region, and the two negative ends fall in different regions. Without loss of generality, let us assume that $y_1$ is in the Dehn twist region and $y_2$ is in the strict Morse region. Again, consider the transverse intersection $\Gamma = C\cap \{\mu=1+\epsilon\}$ for some small positive $\epsilon$. Clearly we have $\int_\Gamma dg_{m,n} = -n$. The local energy inequality then implies that
\[
\int_\Gamma \lambda_{S^*S^n}\geq \epsilon \epsilon_G n.
\]
On the other hand, we must have $0\leq E(C\cap\{x\geq 1+\epsilon\})$, which is to say
\[
0\leq -\int_\Gamma \mu\lambda_{S^*S^n} - \int_\Gamma f_1dg_{m,n} - n f_1(y_2).
\]
So we have $f_1(y_2)\leq -\epsilon_G(\epsilon+\frac{1}{2}\epsilon^2)$, which is a contradiction if we choose $f_1$ to be a sufficiently small perturbation of $\frac{1}{2}\epsilon_G\epsilon^2$ in the Morse region.
\end{proof}

\begin{lemma}
If $x$ is in the Morse region, and at least one of the two negative ends is strictly in the twist region, then $\mathcal{M}(x;y_1,y_2)$ is empty.
\end{lemma}

\begin{proof}
For the first case, assume both $y_1, y_2$ are in the strict Dehn twist region.
Again, if such a $J$-holomorphic section exists, consider the nonempty transverse intersection $\Gamma = C\cap \{\mu=1-\epsilon\}$ for some small positive $\epsilon$. Clearly we have $\int_\Gamma dg_{m,n} = m+n$. The local energy inequality then implies that
\[
\int_\Gamma \lambda_{S^*S^n}\geq \epsilon \epsilon_G\int_\Gamma dg_{m,n} = (m+n)\epsilon \epsilon_G.
\]
On the other hand, we must have $0\leq E(C\cap\{\mu\geq 1-\epsilon\})$, which is to say
\[
0\leq -\int_\Gamma \mu\lambda_{S^*S^n} - \int_\Gamma f_1dg_{m,n} + (m+n) f_1(x).
\]
So we have $f_1(x)\geq \epsilon_G(\epsilon-\frac{1}{2}\epsilon^2)$, which is a contradiction if we choose $f_1$ to be a sufficiently small perturbation of $\frac{1}{2}\epsilon_G\epsilon^2$ in the Morse region.

Finally, suppose that $x$ is in the Morse region and the two negative ends fall in different regions. Without loss of generality, let us assume that $y_1$ is in the strictly twist region and $y_2$ is in the Morse region. Again, consider the transverse intersection $\Gamma = C\cap \{\mu=1-\epsilon\}$ for some small positive $\epsilon$. Clearly we have $\int_\Gamma dg_{m,n} = m$. The local energy inequality then implies that
\[
\int_\Gamma \lambda_{S^*S^n}\geq \epsilon \epsilon_G m.
\]
On the other hand, we must have $0\leq E(C\cap\{x\geq 1-\epsilon\})$, which is to say
\[
0\leq -\int_\Gamma \mu\lambda_{S^*S^n} - \int_\Gamma f_1dg_{m,n} + (m+n)f_1(x)-nf_1(y_2).
\]
So we have $(m+n)f_1(x)-nf_1(y_2)\geq m\epsilon_G(\epsilon-\frac{1}{2}\epsilon^2)$, which is a contradiction if we choose $f_1$ to be a sufficiently small perturbation of $\frac{1}{2}\epsilon_G\epsilon^2$ in the Morse region.
\end{proof}

With the above we have constructed almost complex structures on $E^{m,n}$ that produce confinement of $J$-holomorphic curves satisfying the listed conditions in Theorem \ref{thm:product_confinement} in the purely Morse-Bott setting. The main conditions we require on $J$ are stated in the (proof) of Proposition \ref{prop:morse_bott_all_morse}.

\subsection{Confinement for the nondegenerate setting}
Now that we have thoroughly understood the confinement of $J$-holomorphic curves in the Morse-Bott symplectic fibration $E^{m,n}\rightarrow B_3$, we use a degeneration argument to deduce the same result for sufficiently small perturbations $E^{m,n}_\delta \rightarrow B_3$. Take any almost complex structure $J$ constructed as in Section \ref{sec:confinement_Morse_Bott} for which confinement of $J$-holomorphic curves is achieved. This involves having chosen parameters $\epsilon,\epsilon_G,C$. We note near the punctures of $B_3$ the almost complex structure $J$ can be identified with almost complex structures on bundles $\mathbb{R}\times \mathcal{E}^{m}, \mathbb{R}\times \mathcal{E}^{n}, \mathbb{R}\times \mathcal{E}^{m+n}$ of confinement type. Tacitly we assume our $J$ and parameters $C,\epsilon,\epsilon_G$ are chosen so that confinement of $J$-holomorphic curves as stated in Theorem \ref{thm:confinement_differential} on these bundles also holds. Technically speaking Theorem \ref{thm:confinement_differential} is stated for the nondegenerate case, but here everything is Morse-Bott, but the same theorem continues to hold (we can take $\delta=0$ in all of the proofs).

Then let $J_\delta$ be a size-$\delta$ perturbation of $J$ which is a compatible almost complex structure on $E_{\delta}^{m,n}$.

\begin{proposition}\label{prop:confinement_compatible_Morse}
For $\delta$ sufficiently small, let $u\in \mathcal{M}^{J_{\delta}}(x;y_1,y_2)$ be a $J_\delta$-holomorphic section of $E_\delta^{m,n}$. Then $u$ satisfies the confinement conditions stated in Theorem \ref{thm:product_confinement}.
\end{proposition}

In preparation, we need an estimate of the energy of periodic orbits.

\begin{proposition}
Consider $\mathcal{E}^m$ with vertical $1$-form $\alpha^m$ defined as above. Then the periodic orbits at the $\mu=1$ locus have the lowest action of all its periodic orbits.
\end{proposition}

\begin{proof}
We examine one by one the actions of periodic orbits in the Morse region and in the Dehn twist region.

For the Morse region, the vertical $1$-form looks like $\lambda ^m =\theta_M + mf_1 dt$. For $\gamma$ a periodic orbit in the Morse region, we have
\[
\mathcal{A}(\gamma) = \int_\gamma mf_1dt 
\]
which is just the value of $mf_1$. The value of $f_1$ is zero at the $\mu=1$ locus and is strictly positive in the Morse region, hence the conclusion holds.

Next, we examine the Dehn twist region. For any periodic orbit $\gamma$ that is over some $\mu\in(0, 1)$ with $f_1'(\mu) = -\ell/m$, the tangent vector of $\gamma$ is the horizontal lift of $\partial_t$, which is 
\[
\partial_t^\# = \partial_t + \ell R_P
\]
where $R_P$ denotes the Reeb vector field of the contact manifold $(S^*S^n, \lambda_{S^*S^n})$. It follows that
\[
\mathcal{A}(\gamma) = \int_\gamma \mu\lambda_{S^*S^n}+mf_1(\mu)dt = \ell\mu + mf_1(\mu) = m(f_1(\mu)-\mu f_1'(\mu))
\]
Finally, suppose $m$ is even and the periodic orbit is over $\mu=0$, then the action is
\[
\mathcal{A}(\gamma) = \int_\gamma \mu \lambda_{S^*S^n} + H(\mu) dt = H(0) = -m\tilde{G}_1(0),
\]
where $H(\mu) = \frac{m\mu}{2} -m\tilde{G}_1(\mu)$.

From this, we see that orbits to the left of $\mu=1$ have strictly positive action.
\end{proof}

With the action estimates in mind, we now prove 
\begin{proof}[Proof of 
Proposition \ref{prop:confinement_compatible_Morse}] Pick a sequence $\delta_{i} >0 $ so that $\delta_i\rightarrow 0$ as $i \to +\infty$, and suppose $u_i$ is a sequence of $J_{\delta_{i}}$-holomorphic sections that violates the confinement. Then, in the $i\to\infty$ limit, $u_i$ converges to a cascade \cite{Bourphd,Yaocas}. There is a main level $u^0:B_3\rightarrow E^{m,n}$ and above the main level are $k_\infty$ different $J$-holomorphic sections $u^{1},u^{2},...,u^{k_\infty}$ of $\mathcal{E}^{m+n}$ that are connected by gradient flows of the Morse functions we use to perturb away the Morse-Bott degeneracy. Corresponding to each puncture of $B_3$ at $z=1$ and $z=0$, we get further cascade levels. At the $z=1$ puncture, we have a sequence $u_{1,1},...,u_{1,k_1}$ that are sections of $\mathcal{E}^n$ and they are connected by Morse flow lines; similarly near the $z_0$ puncture, we have $u_{0,1},...,u_{0,k_0}$ as sections of $\mathcal{E}^m$, which are also connected by Morse flow lines.

Note that each component of the cascade satisfies confinement, and the Morse flowlines that connect different cascade levels also satisfy confinement. Since $u_i$ converges in $C^\infty_{loc}$ to the cascade, the only $u_i$ that can violate confinement has the following configuration:

Some upper level of the cascade is contained in the Morse region (resp. Dehn twist region) but has a negative puncture on the $\mu=1$ locus, and some level below this has a positive puncture on the $\mu=1$ locus and leaves the Morse region (resp. Dehn twist region). Notice that in formulating this statement, we have used the fact that all components of the cascade have one positive puncture.

But using action arguments, we see immediately this is not possible: if a component of the cascade has a positive puncture at $\mu=1$, that component itself has negative punctures also at $\mu=1$ since $\mu=1$ has the lowest action of all orbits. This section must then be horizontal. From confinement in the Morse-Bott setting, we see that this component intersects neither the strict Dehn twist region nor the strict Morse region.
\end{proof}

Finally, we show that for tame almost complex structures sufficiently close to a compatible almost complex structure as above, all the relevant holomorphic curves satisfy confinement. We assume $\delta$ is chosen small enough so that all the periodic orbits that arise from perturbing the $\mu=1$ Morse-Bott locus still have the lowest action compared to all other orbits and that any curve (be it cylinder or pair of pants) that has a positive puncture at the orbits corresponding to the $\mu=1$ locus is entirely confined to the intersection region.

\begin{theorem} \label{thm:tame_to_compatible}
Suppose $J_\delta$ is a compatible almost complex structure as constructed in Proposition \ref{prop:confinement_compatible_Morse} for which confinement of $J_\delta$- holomorphic curves is achieved. We assume $\delta$ is small enough as above. For all sufficiently close tame almost complex structures $J'$, the $J'$-holomorphic curves also satisfy confinement.
\end{theorem}

\begin{proof}
Suppose not: then we have a sequence of tame almost complex structures $J_{i}'\rightarrow J$, and a sequence of $J_{i}'$-holomorphic sections $u_i$ that violate confinement, converging to an SFT building $u_\infty$ \cite{SFT_compactness}. However, since each level of $u_\infty$ satisfies confinement as above, then the only possibility is if a level of $u_\infty$ is a curve $v$ which has a negative puncture on an orbit coming from the $\mu=1$ locus: but then all levels below the curve $v$ must stay in the intersection region for energy reasons. This means confinement holds for the entire building $u_\infty$ and hence for $u_i$ for all large enough $i$.
\end{proof}
\subsection{Product for wrapped Dehn twists}
We give a sketch of how to define the appropriate product structure for fixed point Floer cohomology of wrapped Dehn twists. The main objective is to define it so that the curves enumerated by the product with ends in the twist region remain completely confined to the twist region. The details will appear in forthcoming work.

If both $m,n$ are even, then the product on $HF^*(\phi^m_\infty)\otimes HF^*(\phi_\infty^n)\rightarrow HF^*(\phi^{m+n}_\infty)$ is just the product in symplectic cohomology, which is described in detail in \cite{ganatra_thesis}.

In the general case, let $Y_{\phi_\infty^1}\rightarrow S^1$ denote the stable Hamiltonian structure constructed for $\phi^\infty_1$. We assume $\phi$ and the Hamiltonian $H$ are both time independent\footnote{In the Dehn twist region the stable Hamiltonian structure is given by $G_1(\mu)$ without any time-dependent perturbation.}, so the periodic orbits are Morse-Bott. Let $g_{m,n}:B_3\rightarrow S^1$ be the same map described before, and consider the pullback of $Y_{\phi_\infty^1}\rightarrow S^1$ under this map, which we denote by $E^{m,n}_\infty$. We choose a compatible almost complex structure on $E^{m,n}_\infty$ of confinement type near the Dehn twist region, and choose it to be $c$-rescaled contact type near the wrapping region (i.e. we specify it is $c$-rescaled contact type on the fiber, then lift it to the symplectic fibration). Here, the $c$ is fixed by the function $g_{m,n}$.

Using the rescaling functions described in \cite{JYZ} (see also \cite{ganatra_thesis}), the stable Hamiltonian structures near the punctures of $B_3$ of $E_\infty^{m,n}$ can be identified with the fixed point Floer cohomology of $\phi_\infty^n,\phi_\infty^m,\phi_\infty^{m+n}$ respectively. 

The setup is still completely Morse-Bott, and since we have defined it by pullback, the curvature of this fibration is zero. Hence Lemma \ref{lem:confine_product_all_twist} shows that curves in the Dehn twist region are confined to that region. We then perturb the almost complex structure to be tame and perturb the symplectomorphism by a time-dependent Hamiltonian to break the Morse-Bott degeneracy and achieve the necessary transversality. For small enough such perturbations, one shows as in Theorem \ref{thm:tame_to_compatible} that confinement still holds for curves in the Dehn twist region, and one defines the product by counts of pairs of pants in the perturbed fibration. To show this product is well defined after the perturbation, we need to use the dissipative estimates in \cite{ganatra_thesis} to prevent curves from escaping to infinity, which we leave for future work.

\section{Computation: Dehn twists in dimension $4$}

We finish with an example in dimension $4$, demonstrating how our confinement theorems yield concrete computational results:

\begin{proposition}
Let $(M^4,d \theta_M)$ be a $4$-dimensional Liouville domain and let $\phi$ be a Dehn twist around a Lagrangian $2$-sphere $L$ in $M$. Then for $n\geq1$, we have an isomorphism of $\ZZ_2$-modules:
\[
HF^*(M, \phi ^{2n};\ZZ_2) \cong H^*(M\setminus L; \ZZ_2) \oplus \ZZ_{2}^2 \oplus \ZZ_{2}^{4(n-1)}\]
\end{proposition}
where $H^{\ast}(M \setminus L; \ZZ_2)$ denotes the Morse cohomology of the complement of the Lagrangian sphere.
\begin{proof}
By Remark \ref{rmk_hamiltonian}, for even iterations of Dehn twists, the Dehn twist region can be identified with a finite sector of symplectic cohomology of $T^*S^2$ as described in \cite{Lisi_Diogo_complement}. In other words, the differential in the twist region is computing a finite sector of symplectic cohomology with a Hamiltonian that grows quadratically at infinity. We may slightly perturb our Hamiltonian so that it belongs to the class used by \cite{Lisi_Diogo_complement}. We then need to perturb the Hamiltonian in a time-dependent way to achieve nondegeneracy, and we may use the same perturbation in \cite{Lisi_Diogo_complement}. We note that with $\ZZ_2$ coefficients, for each grading degree there are as many generators of the cochain complex in that degree as the rank of the symplectic cohomology group, which means that for $\ZZ_2$ coefficients the differential must vanish in the Dehn twist region. We include the generators at the $\mu=1$ locus as part of the Morse region, and near the boundary the collar neighborhood takes the form $(\epsilon_b/2,\epsilon_b]_\mu\times P$: if we take $f_1=\theta \mu$ for $1\gg\theta>0$, this will ensure we are computing the  \emph{cohomology} of $M\setminus L$.

We have two generators on the zero section, and for each two twists we get four extra generators from Morse-Bott type perturbations; hence, by counting the generators we obtain the claimed result.
\end{proof}

Computations of the product structure in more general settings will be the focus of future work.

% \nocite{*}
\bibliography{references}{}
\bibliographystyle{amsalpha}

\end{document}